\documentclass[a4paper, 12pt, reqno,final]{amsart}
\usepackage{amsaddr} %Adresse auf erste Seite

\usepackage[utf8x]{inputenc}
\usepackage[T1]{fontenc}
\usepackage[english]{babel}
\usepackage{amsmath, amsfonts, amsthm, amssymb, amscd}
\usepackage[final]{graphicx}
\usepackage[below]{placeins}
\usepackage{subfig}
\usepackage{multirow}
\usepackage{mathtools}

\usepackage[hidelinks]{hyperref}
\hypersetup{
  pdfauthor = {Stefan Metzger},
  pdftitle = {Dynamic Boundary SAV}
}

\usepackage[final]{showkeys} %HIER AENDERN
\renewcommand\showkeyslabelformat[1]{}
\usepackage{esint}
\usepackage{caption}
\usepackage{dsfont}

\usepackage{tikz}
\usetikzlibrary{patterns}
\usetikzlibrary{shapes.geometric}

\usepackage[]{todonotes}
\usepackage{diagbox}

\newtheorem{theorem}{Theorem}[section]
\newtheorem{lemma}[theorem]{Lemma}
\newtheorem*{lemma*}{Lemma}
\newtheorem{corollary}[theorem]{Corollary}

\newtheorem{remark}[theorem]{Remark}

\numberwithin{equation}{section}



\makeatletter
\newcommand{\labitem}[2]{%
\def\@itemlabel{\textbf{#1}}
\item
\def\@currentlabel{#1}\label{#2}}
\makeatother
\newcommand{\norm}[1]{\left\|{#1}\right\|}
\newcommand{\abs}[1]{\left|{#1}\right|}

\newcommand{\rkla}[1]{{\left(#1\right)}}
\newcommand{\trkla}[1]{{(#1)}}
\newcommand{\gkla}[1]{{\left\{#1\right\}}}
\newcommand{\tgkla}[1]{{\{#1\}}}

\newcommand{\ekla}[1]{{\left[#1\right]}}
\newcommand{\tekla}[1]{{[#1]}}
\newcommand{\tabs}[1]{|{#1}|}

\newcommand{\bs}[1]{\boldsymbol{#1}}

\newcommand{\iOmega}{\int_{\Omega}}
\newcommand{\iOmegaT}{{\int_0^T\!\!\!\!\int_{\Omega}}}
\newcommand{\iGamma}{\int_{\Gamma}}
\newcommand{\iGammaT}{\int_0^T\!\!\!\!\int_{\Gamma}}

\newcommand{\para}[1]{\partial _{#1}}
\newcommand{\dtau}{\para{\tau}^-}
\newcommand{\ddt}{\frac{\mathrm{d}}{\mathrm{d}t}}

\newcommand{\dx}{\; \mathrm{d}\bs{x}}
\newcommand{\dt}{\; \mathrm{d}t}

\newcommand{\dG}{\; \mathrm{d}\Gamma}

\renewcommand{\div}{\operatorname{div}}
\newcommand{\nablaG}{\nabla_\Gamma}
\newcommand{\DeltaG}{\Delta_\Gamma}

\newcommand{\nn}{^{n}}

\newcommand{\no}{^{n-1}}

\newcommand{\tl}{^{\tau}}
\newcommand{\tp}{^{\tau,+}}
\newcommand{\tm}{^{\tau,-}}
\newcommand{\tpm}{^{\tau,(\pm)}}

\newcommand{\h}{_{h}}
\newcommand{\hx}{_{h,\xi}}

\newcommand{\weak}{\rightharpoonup}
\newcommand{\weakstar}{\stackrel{*}{\rightharpoonup}}

\newcommand{\UhO}{U\h^\Omega}
\newcommand{\UhG}{U\h^\Gamma}

\newcommand{\Th}{\mathcal{T}_h}
\newcommand{\ThG}{\mathcal{T}_h^\Gamma}

\newcommand{\diam}{\operatorname{diam}}

\newcommand{\Ihop}{\mathcal{I}_h}
\newcommand{\Ih}[1]{\Ihop\gkla{#1}}
\newcommand{\IhGop}{\mathcal{I}_h^\Gamma}
\newcommand{\IhG}[1]{\IhGop\gkla{#1}}

\newcommand{\restr}[2]{\ensuremath{% we make the whole thing an ordinary symbol
  \left.\kern-\nulldelimiterspace % automatically resize the bar with \right
  #1 % the function
  \vphantom{\big|} % pretend it's a little taller at normal size
  \right|_{#2} % this is the delimiter
  }}
  
\newcommand{\extend}[2]{\ensuremath{% we make the whole thing an ordinary symbol
  \left.\kern-\nulldelimiterspace % automatically resize the bar with \right
  #1 % the function
  \vphantom{\big|} % pretend it's a little taller at normal size
  \right|^{#2} % this is the delimiter
  }}
  
\newcommand{\mobO}{m}
\newcommand{\mobG}{m_{\Gamma}}

\DeclareMathOperator*{\esssup}{ess\,sup}

\DeclareMathOperator*{\Span}{span}

\graphicspath{{./pictures/}}

\newcounter{IMASTYLE}
\begin{document}
\title[]{A convergent SAV scheme for Cahn--Hilliard equations with dynamic boundary conditions}
\date{\today}
\author[S.~Metzger]{Stefan Metzger}
\address{Friedrich--Alexander Universität Erlangen--Nürnberg,~Cauerstraße 11,~91058~Erlangen,~Germany}
\email{stefan.metzger@fau.de}

\begin{abstract}
The Cahn--Hilliard equation is one of the most common models to describe phase separation processes in mixtures of two materials. For a better description of short-range interactions between the material and the boundary, various dynamic boundary conditions for this equation have been proposed. 
Recently, a family of models using Cahn--Hilliard-type equations on the boundary of the domain to describe adsorption processes was analysed (cf. Knopf, Lam, Liu, Metzger, ESAIM:~Math.~Model.~Numer.~Anal., 2021).
This family of models includes the case of instantaneous adsorption processes studied by Goldstein, Miranville, and Schimperna (Physica D, 2011) as well as the case of vanishing adsorption rates which was investigated by Liu and Wu (Arch. Ration. Mech. Anal., 2019).
In this paper, we are interested in the numerical treatment of these models and propose an unconditionally stable, linear, fully discrete finite element scheme based on the scalar auxiliary variable approach.
Furthermore, we establish the convergence of discrete solutions towards suitable weak solutions of the original model.
Thereby, when passing to the limit, we are able to remove the auxiliary variables introduced in the discrete setting completely.
Finally, we present simulations based on the proposed linear scheme and compare them to results obtained using a stable, non-linear scheme to underline the practicality of our scheme.
\end{abstract}

%35Q35   	PDEs in connection with fluid mechanics
%65M60   	Finite elements, Rayleigh-Ritz and Galerkin methods, finite methods
%65M12   	Stability and convergence of numerical methods
%35G20   	Nonlinear higher-order equations
%35G31   	Initial-boundary value problems for nonlinear higher-order equations
%
%
\selectlanguage{english}

\maketitle
\section{Introduction}

The Cahn--Hilliard equation, which was originally introduced by \ifnum\value{IMASTYLE}>1 {\cite{Cahn1958}\else {Cahn and Hilliard \cite{Cahn1958}\fi, serves in many different applications.
In addition to their original purpose, the description of phase separation and de-mixing processes in binary alloys, Cahn--Hilliard-type equations are also used to model phenomena arising in material sciences, life sciences and image processing.
As certain applications like contact line problems in hydrodynamical applications require an accurate description of the short-range interactions of the binary mixture with the solid wall of the container, several dynamic boundary conditions have been proposed and investigated in the literature.\\
In an open domain $\Omega\subset\mathds{R}^d$ ($d\in\tgkla{2,3}$) with boundary $\Gamma:=\partial\Omega$, the Cahn--Hilliard equation with the classical homogeneous Neumann boundary conditions reads
\begin{subequations}\label{eq:cahn-hilliard}
\begin{align}
\para{t}\phi&=\mobO\Delta\mu&&\text{in~}\Omega\times\trkla{0,T}\,,\\
\mu&=-\varepsilon\sigma\Delta\phi+\varepsilon^{-1}\sigma {F}^\prime\trkla{\phi}&&\text{in~}\Omega\times\trkla{0,T}\,,\label{eq:cahn-hilliard:mu}\\
\nabla\mu\cdot\bs{n}&=0&&\text{on~}\Gamma\times\trkla{0,T}\,,\label{eq:CH:bc:noflux}\\
\nabla\phi\cdot\bs{n}&=0&&\text{on~}\Gamma\times\trkla{0,T}\,,\label{eq:CH:bc:angle}
\end{align}
\end{subequations}
in combination with suitable initial conditions for the phase-field parameter $\phi$.
Here, the parameter $\sigma$ is related to the surface tension, $\mobO$ denotes the mobility constant in the bulk, $\varepsilon>0$ is a small parameter related to the thickness of the interfacial region, and the chemical potential $\mu$ is given as the first variation of the free energy
\begin{align}\label{eq:EOmega}
\mathcal{E}_\Omega\trkla{\phi}:=\varepsilon\sigma\iOmega\tfrac12\tabs{\nabla\phi}^2\dx+\varepsilon^{-1}\sigma\iOmega F\trkla{\phi}\dx\,,
\end{align}
where $F$ is a double-well potential with minima in (or close to) $\phi=\pm1$ representing the pure fluid phases.
Typical choices for this double-well potential are the logarithmic double-well potential
\begin{align}
W_{\log}\trkla{\phi}:=\tfrac{\vartheta}2\trkla{1+\phi}\log\trkla{1+\phi}+\tfrac{\vartheta}2\trkla{1-\phi}\log\trkla{1-\phi}-\tfrac{\vartheta_c}2\phi^2
\end{align}
with $0<\vartheta<\vartheta_c$, the double obstacle potential
\begin{align}
W_{\operatorname{obst}}\trkla{\phi}:=\left\{\begin{matrix}
\vartheta\trkla{1-\phi^2}&\phi\in\tekla{-1,+1}\\\infty&\text{else}
\end{matrix}\right.&&\text{with~}\vartheta>0\,,
\end{align}
and the polynomial double-well potential $W_{\operatorname{pol}}\trkla{\phi}:=\tfrac14\trkla{\phi^2-1}^2$, which in practical computations may be enhanced by an additional penalty term of the form $C_{\operatorname{pen}}\max\gkla{\tabs{\phi}-1,\,0}^2$ with $C_{\operatorname{pen}}>\!\!\!>1$ forcing $\phi$ to stay closer to the physically relevant interval $\tekla{-1,+1}$.
Cahn--Hilliard equations with polynomial double-well potentials are investigated, e.g.~in \cite{Elliott86, Zheng1986, Pego89, BatesFife93, Grinfeld1995, RybkaHoffmann99}.
For Cahn--Hilliard equations with the singular potentials $W_{\operatorname{log}}$ and $W_{\operatorname{obst}}$, we refer the reader to \cite{Blowey1991, Elliott1996, AbelsWilke2007, Cherfils2011}.

The boundary condition \eqref{eq:CH:bc:noflux} states that there is no flux across $\Gamma$, i.e.~$\iOmega\phi\dx$ is conserved, while condition \eqref{eq:CH:bc:angle} indicates that the zero level set of $\phi$ representing the fluid-fluid interface intersects $\Gamma$ at a static contact angle of $\tfrac\pi2$.
This can be interpreted as negligence of the interactions between the fluids and the walls of the surrounding container.
To improve description of the interactions between the solid wall and the fluids, physicists suggested to add a surface free energy 
\begin{align}\label{eq:EGamma}
\mathcal{E}_\Gamma\trkla{\phi}=\delta\kappa\iGamma\tfrac12\tabs{\nablaG\phi}^2\dG+\delta^{-1}\iGamma {G}\trkla{\phi}\dG
\end{align}
which is also of Ginzburg--Landau type (cf.~\cite{Fischer1997,Fischer1998a,Kenzler2001}).
Here, $\nablaG$ denotes the surface gradient operator on $\Gamma$ and ${G}$ denotes a suitable boundary potential which might be chosen similar to ${F}$.
The parameter $\kappa\geq0$ describes the influence of surface diffusion and $\delta>0$ is related to the thickness of the interfacial regions on the boundary.
Dynamic boundary conditions that ensure that the total energy $\tekla{\mathcal{E}_\Omega+\mathcal{E}_\Gamma}$ is not increasing in time are for instance Allen--Cahn-type boundary conditions (cf. \cite{Fischer1997, Fischer1998, RackeZheng2003, Wu2004, ChillFasangovaPruess2006, Gilardi2009, Miranville2010, Cherfils2011, Liero2013, Colli2014, Colli2015, Mininni2017}), where \eqref{eq:CH:bc:angle} is replaced by
\begin{subequations}\label{eq:bc:AC}
\begin{align}
\para{t}\phi&=-\mobG\theta&&\text{on~}\Gamma\times\trkla{0,T}\,,\\
\theta&=-\kappa\delta\Delta_\Gamma\phi+\delta^{-1}{G}^\prime\trkla{\phi}+\varepsilon\sigma\nabla\phi\cdot\bs{n}&&\text{on~}\Gamma\times\trkla{0,T}\,.
\end{align}
\end{subequations}
Here, $\mobG$ is the mobility constant on the boundary and $\Delta_\Gamma$ denotes the Laplace-Beltrami operator describing surface diffusion on $\Gamma$.
In the case $\kappa=0$ this problem is related to the moving contact line problem discussed in \cite{Qian2006}.
In this work, however, we shall always assume $\kappa=1$.
Other approaches additionally replaced \eqref{eq:CH:bc:noflux} by enforcing that $\theta$ equals the trace of $\mu$ (cf.~\cite{Gal2006}).\\
In addition to Allen--Cahn-type dynamic boundary conditions, different Cahn--Hilliard-type boundary conditions have also been proposed (see e.g.~\cite{Goldstein2011, Liu2019, KnopfLamLiuMetzger2021}).
In this publication, we discuss the numerical treatment of the following phase-field model with rate dependent dynamic boundary conditions which was introduced in \cite{KnopfLamLiuMetzger2021}:
\begin{subequations}\label{eq:model}
\begin{align}
\para{t}\phi&=\mobO\Delta\mu&&\text{in~} \Omega\times\trkla{0,T}\,,\\
\mu&=-\varepsilon\sigma\Delta\phi+\varepsilon^{-1}\sigma {F}^\prime\trkla{\phi}&&\text{in~} \Omega\times\trkla{0,T}\,,\\
\para{t}\phi&=\mobG\DeltaG\theta -\beta\mobO\nabla\mu\cdot\bs{n}&&\text{on~}\Gamma\times\trkla{0,T}\,,\\
\theta&=-\delta\DeltaG\phi+\delta^{-1}{G}^\prime\trkla{\phi}+\varepsilon\sigma\nabla\phi\cdot\bs{n}&&\text{on~}\Gamma\times\trkla{0,T}\,,\\
\nabla\mu\cdot\bs{n}&=\xi\trkla{\beta\theta-\mu}&&\text{on~}\Gamma\times\trkla{0,T}\,,\label{eq:model:jump}\\
\phi\trkla{0}&=\phi_0&&\text{in~}\overline{\Omega}\,.
\end{align}
\end{subequations}
Here, $\xi\in\trkla{0,\infty}$ can be interpreted as an adsorption rate.
In the case of instantaneous adsorption, i.e.~$\xi=\infty$, \eqref{eq:model:jump} is understood as $\beta\theta\equiv\mu$ on $\Gamma\times\trkla{0,T}$, which allows us to recover the model proposed in \cite{Goldstein2011}.
There, the potential on the boundary $\theta$ coincides with the trace of the bulk potential $\mu$ up to a weight function $\beta> 0$ which we will assume to be constant.
For $\xi=0$, on the other hand, we recover the model proposed in \cite{Liu2019} describing the scenario of a vanishing adsorption rate.
In both cases, the total mass of $\phi$ is conserved, i.e.
\begin{align}
\ddt\ekla{\iOmega\phi\dx+\beta^{-1}\iGamma\phi\dG}=0\,.
\end{align}
In the case of vanishing adsorption rates ($\xi=0$), however, mass transfer between the bulk and the boundary is prohibited, resulting additionally in individual conservation of $\phi$, i.e.
\begin{align}
\ddt\iOmega\phi\dx=0&&\text{and}&&\ddt\iGamma\phi\dG=0\,.
\end{align}
In the case $\xi=0$, the well-posedness of \eqref{eq:model} was studied in \cite{Liu2019} and \cite{GarckeKnopf20}.
The long-time behavior of this scenario was investigated in \cite{Liu2019} and \cite{MiranvilleWu20}.
Numerical analysis and simulations for this case of vanishing adsorption rates can be found in \cite{GarckeKnopf20}, \cite{Metzger2021a}, and \cite{BaoZhang21a}.\\
The well-posedness of \eqref{eq:model} in the general case $\xi\in\tekla{0,\infty}$ and the asymptotic limits $\xi\searrow0$ and $\xi\nearrow\infty$ were discussed in \cite{KnopfLamLiuMetzger2021}.
The long term behavior of this general case was studied in \cite{GarckeKnopfYayla2022}.
For the numerical treatment and the convergence properties of discrete solutions, we refer to \cite{KnopfLamLiuMetzger2021} and \cite{BaoZhang21b}.\\
Other variants of dynamic boundary conditions include Robin-type conditions for $\phi$ (cf.~\cite{KnopfLam20}) or non-local Cahn--Hilliard-type equations (cf.~\cite{KnopfSignori21}).\\

In this publication, we discuss the numerical treatment of \eqref{eq:model} and propose a new fully discrete, convergent finite element scheme based on the scalar auxiliary variable approach.
The numerical treatment of plain Cahn--Hilliard equations \eqref{eq:cahn-hilliard} and theirs variants -- often in combination with Navier--Stokes-equations -- was intensely discussed through the recent years, resulting in various different discretization techniques.
Although, we will focus on non-singular potentials, we do not want to conceal that there are also numerical schemes at hand which are able to deal with the singular potentials $W_{\operatorname{log}}$ and $W_{\operatorname{obst}}$ (see e.g.~\cite{CopettiElliott92,Blowey1992, Blowey1996,Barrett1999,Barrett2001, Frank2020}).\\
Discretization techniques, which transfer the energy stability of the Cahn--Hilliard equation to a discrete setting without imposing restrictions on the size of the time increment $\tau>0$, include approaches based on convex-concave splittings of the energy or the polynomial double-well potential (cf.~\cite{Elliott1993, Eyre98, WiseWangLowengrub2009, ShenWangWangWise2012} and  \cite{Grun2013,Grun2013c,LiuShenYang2015,GarckeHinzeKahle2016, GrunGuillenMetzger2016} for an application to Cahn--Hilliard--Navier--Stokes-systems) and stabilized linearly implicit approaches (cf.~\cite{XuTang2006,ShenYang2010,ShenYang2010b,Nochetto2016, BaoZhang21a,BaoZhang21b}).
While the first one typically results in non-linear discrete systems, the second approach provides discrete systems that are linear with respect to the unknown quantities, but often relies on additional regularizations or large stabilization parameters.
To overcome these issues the method of invariant energy quadratization (IEQ), which is a generalization of the Lagrange multiplier method developed in \cite{Guillen-Tierra13} (see also \cite{Badia2011a}), was proposed (see e.g.~\cite{Yang2016,ChengYangShen2017, Han2017}).
For a positive potential $F$, this method produces linear and unconditionally stable schemes by introducing the auxiliary function $u=\sqrt{F\trkla{\phi}}$.
Replacing \eqref{eq:cahn-hilliard:mu} by
\begin{align}
\mu=-\varepsilon\sigma\Delta\phi+\frac{\sigma}{\varepsilon}\frac{F^\prime\trkla{\phi}}{\sqrt{F\trkla{\phi}}} u
\end{align}
provides more leeway with respect to the time discretization of the non-linear term.
On the other hand, however, one additionally has to track the evolution of $u$ which according to the chain rule is formally given by
\begin{align}\label{eq:evolution:ieq}
\para{t}u=\frac{\mathrm{d}}{\mathrm{d}t} \sqrt{F\trkla{\phi}}=\frac{F^\prime\trkla{\phi}}{2\sqrt{F\trkla{\phi}}}\para{t}\phi\,.
\end{align}

Recently, this technique was enhanced to the scalar auxiliary variable (SAV) approach (see e.g.~\cite{ShenXuYang2018,LiShen2020}).
The SAV approach simplifies the IEQ approach by introducing only the scalar auxiliary variable 
\begin{align}
{r}\trkla{t}:=\sqrt{\iOmega {F}\trkla{\phi}\dx}
\end{align}
to rewrite the non-linear term instead of an auxiliary function.
Therefore, we do not require a positive lower bound pointwise for $F$, but can relax this constraint to a positive lower bound for $\iOmega F\trkla{\phi}\dx$.
More importantly, however, it replaces \eqref{eq:evolution:ieq} by a single scalar valued evolution equation.
In particular, the evolution of the auxiliary variable given by
\begin{align}\label{eq:evolution:aux}
\para{t} {r}= \frac{1}{2\sqrt{\iOmega {F}\trkla{\phi}\dx}}\iOmega {F}^\prime\trkla{\phi}\para{t}\phi\dx\,.
\end{align}
Although IEQ and SAV schemes are very promising as they allow for unconditionally stable linear schemes and are already applied to many intricate problems (see e.g.~\cite{ShenXuYang19, ShenYang20} and the references therein), the additional auxiliary variables complicate the analytical treatment.
While many convergence results provide error estimates under the assumption that a sufficiently regular solution exists (see e.g.~\cite{LinCaoZhangSun20, YangZhang20, ShenXu18, LiShen2020}), convergence results without this assumption are scarce.
Assuming that the initial conditions for the phase-field parameter $\phi$ are in $H^4$, \ifnum\value{IMASTYLE}>1 {\cite{ShenXu18}} \else {Shen and Xu \cite{ShenXu18}}\fi were able to show that subsequences of discrete solutions converge (weakly) towards weak solutions of the continuous problem.
For an investigation of the asymptotic behaviour of SAV schemes, we refer the reader to \cite{BouchritiPierreAlaa20}.\\
From the methods listed above, by now only convex-concave splittings of the potentials (cf.~\cite{Metzger2021a,KnopfLamLiuMetzger2021}) and stabilized linearly implicit approaches (cf.~\cite{BaoZhang21a,BaoZhang21b}) have been applied to \eqref{eq:model}.
Numerical schemes using extrapolation techniques to discretize the non-linear terms were discussed e.g.~in \cite{HarderKovacs21}.
A finite difference model for the treatment of the Allen--Cahn-type dynamic boundary conditions \eqref{eq:bc:AC} was proposed in \cite{Kenzler2001}.
For an application of the IEQ approach to Allen--Cahn-type dynamic boundary conditions, we refer to \cite{YangYu2018}.\\

The outline of the paper is as follows.
In Section \ref{sec:scheme}, we introduce the discrete function spaces, collect our assumptions on the data, present the discrete scheme, and use its energy stability to prove the existence of a unique discrete solution.
In Section \ref{sec:regularity}, we establish improved regularity results, which will be used in Section \ref{sec:limit} to pass to the limit $\trkla{h,\tau,\xi}\rightarrow\trkla{0,0,\xi_0}$ and obtain the existence of a unique weak solution.
We conclude by presenting numerical simulations in Section \ref{sec:numerics}.
Thereby, we consider the test cases discussed in \cite{Metzger2021a} and \cite{KnopfLamLiuMetzger2021} and compare the results of the proposed SAV scheme with the results obtained from the splitting scheme analysed in \cite{Metzger2021a} and \cite{KnopfLamLiuMetzger2021}.\\

\underline{Notation:} Given the spatial domain $\Omega\subset\mathds{R}^d$ with $d\in\tgkla{2,3}$, we denote the space of $k$-times weakly differentiable functions with weak derivatives in $L^p\trkla{\Omega}$ by $W^{k,p}\trkla{\Omega}$.
For $p=2$, we will denote the Hilbert spaces $W^{k,2}\trkla{\Omega}$ by $H^k\trkla{\Omega}$.
For spaces defined on $\Gamma:=\partial\Omega$, we shall use a similar notation.
Throughout this publication, we shall always assume that $\Omega$ has a lipschitzian boundary.
In this case, the spaces $L^p\trkla{\Gamma}$ and $W^{1,p}\trkla{\Gamma}$ ($p\geq1$) are well-defined (cf.~\cite{Kufner77}) and the trace operator $\cdot\big\vert_{\Gamma}$ is uniquely defined and lies in $\mathcal{L}\trkla{W^{1,p}\trkla{\Omega}, W^{1-1/p,p}\trkla{\Gamma}}$ (cf.~\cite{Necas2012}).
For brevity, we will sometimes (in particular when the considered function is continuous) suppress the trace operator and write $v$ instead of $v\big\vert_\Gamma$.
In addition, we define the spaces
\begin{align}
\widetilde{\mathcal{L}}^p:= L^p\trkla{\Omega}\times L^p\trkla{\Gamma}&&\text{and}&&\mathcal{H}^1:= H^1\trkla{\Omega}\times H^1\trkla{\Gamma}
\end{align}
with $p\in\tekla{1,\infty}$.
Note that $\mathcal{H}^1$ is a Hilbert space with respect to the inner product
\begin{align}
\rkla{\trkla{\phi,\psi},\trkla{\zeta,\xi}}_{\mathcal{H}^1}:=\trkla{\phi,\zeta}_{H^1\trkla{\Omega}}+\trkla{\psi,\xi}_{H^1\trkla{\Gamma}}\qquad \text{for all~}\trkla{\phi,\psi},\trkla{\zeta,\xi}\in\mathcal{H}^1
\end{align}
and its induced norm $\norm{\cdot}_{\mathcal{H}^1}:=\sqrt{\trkla{\cdot,\cdot}_{\mathcal{H}^1}}$.
Furthermore, we introduce the following subspaces of $\mathcal{H}^1$:
\begin{align}
\mathcal{V}&:=\gkla{\trkla{\phi,\psi}\in\mathcal{H}^1\,:\,\phi\big\vert_{\Gamma}=\psi\text{~a.e.~on~}\Gamma}\,,\\
\mathcal{W}_{\beta,m}&:=\gkla{\trkla{\phi,\psi}\in\mathcal{V}\,:\,\iOmega\phi\dx+\beta^{-1}\iGamma\psi\dG=m}
\end{align}
for any $m\in\mathds{R}$ and $\beta\in\mathds{R}^+$.
Endowed with the inner product $\trkla{\cdot,\cdot}_{\mathcal{W}_{\beta,0}}:=\trkla{\cdot,\cdot}_{\mathcal{H}^1}$ and its induced norm, the space $\mathcal{W}_{\beta,0}$ is also a Hilbert space.\\
For a Banach space $X$ and a time interval $I$, the symbol $L^p\trkla{I;X}$ ($p\in\tekla{1,\infty}$) stands for the parabolic space of $L^p$-integrable functions on $I$ with values in $X$ and $H^1\trkla{I;X}$ stands for the space of weakly differentiable functions with weak derivatives in $L^2\trkla{I,X}$.
\section{A stable, fully discrete SAV scheme}\label{sec:scheme}
In this section, we shall propose a stable, fully discrete finite element approximation of \eqref{eq:model}.
In particular, we will discretize the following weak formulation of \eqref{eq:model}:
\begin{subequations}\label{eq:formal}
\begin{align}
\iOmegaT\para{t}\phi \psi\dx\dt + \iOmegaT m\nabla\mu\cdot\nabla\psi\dx\dt -\iGammaT \mobO\xi\trkla{\beta\theta-\mu}\psi\dG\dt=0\,,\label{eq:formal:1}\\
\iGammaT\para{t}\phi\eta\dG\dt+\iGammaT\mobG\nablaG\phi\cdot\nablaG\eta\dG\dt + \beta \mobO\iGammaT\xi\trkla{\beta\theta-\mu}\eta\dG\dt=0\,,\label{eq:formal:2}
\end{align}
\begin{align}
\begin{split}
\iOmegaT\mu\hat{\psi}\dx\dt+\iGammaT\theta\hat{\psi}\dG\dt=\varepsilon\sigma\iOmegaT\nabla\phi\cdot\nabla\hat{\psi}\dx\dt + \delta\iGammaT\nablaG\phi\cdot\nablaG\hat{\psi}\dG\dt\\
+\varepsilon^{-1}\sigma\int_0^T\frac{r\trkla{t}}{\sqrt{\iOmega F\trkla{\phi}\dx}}\iOmega F^\prime\trkla{\phi}\hat{\psi}\dx\dt +\delta^{-1}\int_0^T\frac{s\trkla{t}}{\sqrt{\iGamma G\trkla{\phi}\dG}}\iGamma G^\prime\trkla{\phi}\hat{\psi}\dG\dt
\end{split}
\end{align}
for appropriate test functions $\psi,\eta,\hat{\psi}$.
Here, we introduced the scalar auxiliary variables 
\begin{align*}
r\trkla{t}:=\sqrt{\iOmega F\trkla{\phi\trkla{t}}\dx}&&\text{and}&&s\trkla{t}:=\sqrt{\iGamma G\trkla{\phi\trkla{t}}\dG}\,.
\end{align*}
According to \eqref{eq:evolution:aux}, their evolution is given by
\begin{align}
&&\para{t} r&=\frac{1}{2\sqrt{\iOmega F\trkla{\phi}\dx}}\iOmega F^\prime\trkla{\phi}\para{t}\phi\dx\,,\label{eq:formal:auxr}&&
\text{and}&&\para{t} s&=\frac{1}{2\sqrt{\iGamma G\trkla{\phi}\dG}}\iGamma G^\prime\trkla{\phi}\para{t}\phi\dG.
\end{align}
\end{subequations}

To propose a fully discrete finite element approximation of \eqref{eq:formal}, we shall start by introducing the general assumptions, discrete function spaces and the interpolation operators used in the considered scheme.
Concerning the discretization in time, we assume that
\begin{itemize}
\labitem{\textbf{(T)}}{item:T} the time interval $I:=[0,T)$ is subdivided in intervals $I_{n}:=[t\no,t\nn)$ with $t\nn=t\no+\tau_{n}$ for time increments $\tau_{n}>0$ and $n=1,\ldots,N$ such that $t^N=T$. For simplicity, we take $\tau_n\equiv\tau=\tfrac{T}{N}$ for $n=1,\ldots,N$.
\end{itemize}
The spatial domain $\Omega\subset\mathds{R}^d$ is assumed to be bounded and convex.
Furthermore, we shall assume that $\Omega$ is polygonal (or polyhedral, respectively) to avoid additional technicalities.
The spatial discretization is based on partitions $\Th$ of $\Omega$ and $\ThG$ of $\Gamma:=\partial\Omega$ depending on a discretization parameter $h>0$ satisfying the following assumptions.
\begin{itemize}
\labitem{\textbf{(S1)}}{item:S1} Let $\tgkla{\Th}_{h>0}$ be a quasiuniform (in the sense of \cite{BrennerScott}) family of partitions of $\Omega$ into disjoint, open simplices $K$, satisfying
\begin{align*}
\overline{\Omega}\equiv\bigcup_{K\in\Th}\overline{K}&&\text{with}~\max_{K\in\Th}\diam\trkla{K}\leq h\,.
\end{align*}
\labitem{\textbf{(S2)}}{item:S2} Let $\tgkla{\ThG}_{h>0}$ be a quasiuniform family of partitions of $\Gamma$ into disjoint open simplices $K^\Gamma$, satisfying
\begin{align*}
&&\forall K^\Gamma&\in\ThG~\exists!K\in\Th~\text{such that}~\overline{K^\Gamma}=\overline{K}\cap\Gamma\,,\\
\text{and}&& \Gamma&\equiv\bigcup_{K^\Gamma\in\ThG}\overline{K^\Gamma}\qquad\text{with}~\max\diam\trkla{K^\Gamma}\leq h\,.
\end{align*}
\end{itemize}
The second assumption \ref{item:S2} implies that all elements in $\ThG$ are edges (or faces, respectively) of elements of $\Th$.
For the approximation of the phase-field $\phi$ and the chemical potential $\mu$ we use continuous, piecewise linear finite element functions on $\Th$. This space shall be denoted by $\UhO$ and is spanned by the functions $\tgkla{\chi_{h,k}}_{k=1,\ldots,\dim\UhO}$ forming a dual basis to the vertices $\tgkla{\bs{x}_k}_{k=1,\ldots,\dim\UhO}$ of $\Th$.
Due to the constraints on $\Th$ and $\ThG$, the space of continuous, piecewise linear finite element functions $\UhG$ on $\ThG$ is given by the traces of functions in $\UhO$, i.e.
\begin{align}
\UhG=\Span\tgkla{\psi\h\big\vert_\Gamma\,:\,\psi\h\in\UhO}.
\end{align}
This space can also be described as the span of functions $\tgkla{\chi^\Gamma_{h,k}}_{k=1,\ldots,\dim\UhG}$ which form a dual basis to the vertices $\tgkla{\bs{x}^\Gamma_k}_{k=1,\ldots,\dim\UhG} \subset \tgkla{\bs{x}_j}_{j=1,\ldots,\dim\UhO}$ of $\ThG$.\\
We define the nodal interpolation operators $\Ihop\,:\, C^0\trkla{\overline{\Omega}}\rightarrow \UhO$ and $\IhGop\,:\, C^0\trkla{\Gamma}\rightarrow \UhG$ by
\begin{align}
\Ih{a}:=\sum_{k=1}^{\dim\UhO}a\trkla{\bs{x}_k}\chi_{h,k}\,,&&\text{and}&&\IhG{a}:=\sum_{k=1}^{\dim\UhG}a\trkla{\bs{x}^\Gamma_k}\chi^\Gamma_{h,k}\,.
\end{align}

For future reference, we state the following norm equivalences for $p\in[1,\infty)$ and $f\h\in\UhO$ and $\tilde{f}\h\in\UhG$:
\begin{subequations}\label{eq:normequivalence}
\begin{align}
c\rkla{\iOmega\tabs{f\h}^p\dx}^{1/p}\leq \rkla{\iOmega\Ih{\tabs{f\h}^p}\dx}^{1/p}&\leq C\rkla{\iOmega\tabs{f\h}^p\dx}^{1/p}\,,\\
c\rkla{\iGamma\tabs{\tilde{f}\h}^p\dG}^{1/p}\leq\rkla{\iGamma\IhG{\tabs{\tilde{f}\h}^p}\dG}^{1/p}&\leq C\rkla{\iGamma\tabs{\tilde{f}\h}^p\dG}^{1/p}
\end{align}
\end{subequations}
wit $c,C>0$ independent of $h$.
Furthermore, the following estimates (that can be found in Lemma 2.1 in \cite{Metzger2020}) hold true:
\begin{lemma}\label{lem:Ih}
Let $\Th$ and $\ThG$ satisfy \ref{item:S1} and \ref{item:S2}, respectively.
Furthermore, let $p\in[1,\infty)$, $1\leq q\leq\infty$, and $q^*=\tfrac{q}{q-1}$ for $q<\infty$ or $q^*=1$ for $q=\infty$. Then
\begin{subequations}
\begin{align}
\norm{\trkla{1-\Ihop}\tgkla{f\h g\h}}_{L^p\trkla{\Omega}}&\leq C h^2\norm{\nabla f\h}_{L^{pq}\trkla{\Omega}}\norm{\nabla g\h}_{L^{pq^*}\trkla{\Omega}}\,,\\
\norm{\trkla{1-\Ihop}\tgkla{f\h g\h}}_{W^{1,p}\trkla{\Omega}}&\leq C h\norm{\nabla f\h}_{L^{pq}\trkla{\Omega}}\norm{\nabla g\h}_{L^{pq^*}\trkla{\Omega}}\,,\\
\norm{\trkla{1-\IhGop}\tgkla{\tilde{f}\h \tilde{g}\h}}_{L^p\trkla{\Gamma}}&\leq C h^2\norm{\nablaG \tilde{f}\h}_{L^{pq}\trkla{\Gamma}}\norm{\nablaG \tilde{g}\h}_{L^{pq^*}\trkla{\Gamma}}\,,\\
\norm{\trkla{1-\IhGop}\tgkla{\tilde{f}\h \tilde{g}\h}}_{W^{1,p}\trkla{\Gamma}}&\leq C h\norm{\nablaG \tilde{f}\h}_{L^{pq}\trkla{\Gamma}}\norm{\nablaG \tilde{g}\h}_{L^{pq^*}\trkla{\Gamma}}
\end{align}
hold true for all $f\h,g\h\in\UhO$ and $\tilde{f}\h,\tilde{g}\h\in\UhG$.
\end{subequations}
\end{lemma}

Concerning the potentials $F$ and $G$, we will assume that
\begin{itemize}
\labitem{\textbf{(P)}}{item:potential} $F,\,G\in C^{1}\trkla{\mathds{R}}$ satisfy the following properties:
\begin{itemize}
\item $F$ and $G$ are bounded from below such that
\begin{align*}
E\h^\Omega\trkla{\zeta\h}:=\iOmega\Ih{F\trkla{\zeta\h}}\dx\geq\gamma>0&&\text{and}&&E\h^\Gamma\trkla{\tilde{\zeta}\h}:=\iGamma\IhG{G\trkla{\tilde{\zeta}\h}}\dG\geq\gamma>0
\end{align*}
for all $\zeta\h\in\UhO$ and $\tilde{\zeta}\h\in\UhG$ uniformly in $h$.
\item $F$, $G$, and their first derivatives satisfy the growth estimates
\begin{align*}
\tabs{F\trkla{\zeta}}&\leq C\rkla{1+\tabs{\zeta}^4}\,,&\tabs{G\trkla{\zeta}}&\leq C\trkla{1+\tabs{\zeta}^4}\,,\\
\tabs{F^\prime\trkla{\zeta}}&\leq C\rkla{1+\tabs{\zeta}^3}\,,&\tabs{G^\prime\trkla{\zeta}}&\leq C\trkla{1+\tabs{\zeta}^3}\,.
\end{align*}
\item $F$ and $G$ can be decomposed into $F=F_1+F_2$ and $G=G_1+G_2$ with $F_1,G_1\in C^{1,1}\trkla{\mathds{R}}$ and $F_2,G_2\in C^2\trkla{\mathds{R}}$ satisfying
\begin{align*}
\abs{F^{\prime\prime}_2\trkla{\zeta}}\leq C\trkla{1+\tabs{\zeta}^2}&&\text{and}&&\abs{G^{\prime\prime}_2\trkla{\tilde{\zeta}}}\leq C\trkla{1+\tabs{\tilde{\zeta}}^2}\,.
\end{align*}
\end{itemize}
\end{itemize}
These assumptions are satisfied, e.g., by the shifted polynomial double-well potential~ $\widetilde{W}_{\operatorname{pol}}\trkla{\phi}:=\tfrac14\trkla{1-\phi^2}^2+c_0$ with any $c_0>0$, but not by the singular potentials $W_{\log}$ and $W_{\operatorname{obst}}$.\\
In order to establish uniform improved regularity results in Section \ref{sec:regularity}, we require an additional assumption on the discretization parameters $h$ and $\tau$.
In particular, we will need that
\begin{itemize}
\labitem{\textbf{(C)}}{item:htau1} $\displaystyle\frac{h^4}{\tau}\searrow0$ for $\trkla{h,\tau}\searrow\trkla{0,0}$.
\end{itemize}
%As it will turn out, passing the limit in the discrete version of \eqref{eq:formal:auxr} is intricate due to the unavailability of sufficiently high regularity results.
%We will therefore need an additional, more restrictive constraint on $h$ and $\tau$ when passing to the limit in Section \ref{sec:limit}.
%Depending on the dimension $d$, we need to assume that
%\begin{itemize}
%\labitem{\textbf{(C2)}}{item:htau2} $\displaystyle\frac{\tau}{h^2}\searrow0$ for $\trkla{h,\tau}\searrow\trkla{0,0}$, if $d=3$, and $\displaystyle\frac{\tau}{h^{2\nu/\trkla{2+\nu}}}\searrow 0$ for $\trkla{h,\tau}\searrow\trkla{0,0}$ and any $\nu>0$, if $d=2$.
%\end{itemize}
Concerning the initial data, we will assume that
\begin{itemize}
\labitem{\textbf{(I)}}{item:I} $\phi_0\in H^2\trkla{\Omega}$. The discrete initial data $\phi^0\hx\in\UhO$, $r\hx^0\in \mathds{R}$, and $s\hx^0\in\mathds{R}$ are then obtained via $\phi\hx^0:=\Ih{\phi_0}$, $r\hx^0:=\sqrt{E\h^\Omega\trkla{\phi\hx^0}}$, and $s\hx^0:=\sqrt{E\h^\Gamma\trkla{\phi\hx^0}}$.
\end{itemize}
Immediate consequences of \ref{item:I} are
\begin{subequations}
\begin{align}\label{eq:intialbounds}
\norm{\phi\hx^0}_{W^{1,6}\trkla{\Omega}} +\norm{\phi\hx^0\big\vert_{\Gamma}}_{W^{1,4}\trkla{\Gamma}} + \tabs{r\hx^0}^2+\tabs{s\hx^0}^2\leq C\trkla{\phi_0}
\end{align}
\begin{align}
\text{and}&&\norm{\phi\hx^0-\phi_0}_{H^1\trkla{\Omega}} +\norm{\phi\hx^0\big\vert_\Gamma-\phi_0\big\vert_\Gamma}_{L^4\trkla{\Gamma}}\leq C\trkla{\phi_0}h\,.
\end{align}
\end{subequations}
Denoting the backward difference quotient in time by $\dtau$, we are now in the position to introduce our discrete scheme.
To simplify the notation, we set w.l.o.g.~$\sigma=\mobO=\mobG=\varepsilon=\delta=1$.
In order to obtain a scheme that is not only applicable to $\xi\in\trkla{0,\infty}$, but also covers the limit cases $\xi=0$ and $\xi=\infty$, we follow the lines of \cite{KnopfLamLiuMetzger2021} and replace \eqref{eq:formal:1} by $\tekla{\eqref{eq:formal:1}+\eqref{eq:formal:2}}$ and multiply \eqref{eq:formal:2} by $\tfrac1{1+\xi}$.
The resulting finite element scheme reads:\\

For $\xi\in\tekla{0,\infty}$ and $\trkla{\phi\hx\no,r\hx\no,s\hx\no}\in\UhO\times\mathds{R}\times\mathds{R}$ given, find 
\begin{align*}
\trkla{\phi\hx\nn,\mu\hx\nn,\theta\hx\nn,r\hx\nn,s\hx\nn}\in\UhO\times\UhO\times\UhG\times\mathds{R}\times\mathds{R}
\end{align*}
such that
\begin{subequations}\label{eq:modeldisc}
\begin{multline}\label{eq:modeldisc:phase1}
\iOmega\Ih{\dtau \phi\hx\nn \psi\h}\dx +\iOmega\nabla\mu\hx\nn\cdot\nabla\psi\h\dx +\beta^{-1}\iGamma\IhG{\dtau \phi\hx\nn\psi\h}\dG\\
+\beta^{-1}\iGamma\nablaG\theta\hx\nn\cdot\nablaG\psi\h\dG=0\,,
\end{multline}
\begin{multline}\label{eq:modeldisc:phase2}
\tfrac1{1+\xi}\iGamma\IhG{\dtau\phi\hx\nn\eta\h}\dG+\tfrac1{1+\xi}\iGamma\nablaG\theta\hx\nn\cdot\nablaG\eta\h\dG\\
+\tfrac\xi{1+\xi}\beta\iGamma\IhG{\trkla{\beta\theta\hx\nn-\mu\hx\nn}\eta\h}\dG=0\,,
\end{multline}
\begin{multline}\label{eq:modeldisc:pot}
\iOmega\Ih{\mu\hx\nn\hat{\psi}\h}\dx+\iGamma\IhG{\theta\hx\nn\hat{\psi}\h}\dG=\iOmega\nabla\phi\hx\nn\cdot\nabla\hat{\psi}\h\dx+\iGamma\nablaG\phi\hx\nn\cdot\nablaG\hat{\psi}\h\dG\\
+ \frac{r\hx\nn}{\sqrt{E\h^\Omega\trkla{\phi\hx\no}}}\iOmega \Ih{F^\prime\trkla{\phi\hx\no}\hat{\psi}\h}\dx + \frac{s\hx\nn}{\sqrt{E\h^\Gamma\trkla{\phi\hx\no}}}\iGamma\IhG{G^\prime\trkla{\phi\hx\no}\hat{\psi}\h}\dG\,,
\end{multline}
\begin{align}\label{eq:modeldisc:auxr}
r\hx\nn&=r\hx\no +\frac{1}{2\sqrt{E\h^\Omega\trkla{\phi\hx\no}}}\iOmega\Ih{F^\prime\trkla{\phi\hx\no}\trkla{\phi\hx\nn-\phi\hx\no}}\dx\,,\\
s\hx\nn&=s\hx\no +\frac{1}{2\sqrt{E\h^\Gamma\trkla{\phi\hx\no}}}\iGamma\IhG{G^\prime\trkla{\phi\hx\no}\trkla{\phi\hx\nn-\phi\hx\no}}\dG\,\label{eq:modeldisc:auxs}
\end{align}
\end{subequations}
for all $\psi\h,\hat{\psi}\h\in\UhO$ and $\eta\h\in\UhG$.\\
This fully discrete finite element scheme is linear w.r.t.~the unknowns $\trkla{\phi\hx\nn,\mu\hx\nn,\theta\hx\nn, r\hx\nn,s\hx\nn}\in\UhO\times\UhO\times\UhG\times\mathds{R}\times\mathds{R}$, well-defined for all $\xi\in\tekla{0,\infty}$, and preserves the combined mass conservation
\begin{align}\label{eq:conservation}
\iOmega\phi\hx\nn\dx+\beta^{-1}\iGamma\phi\hx\nn\dG=\iOmega\phi\hx^0\dx+\beta^{-1}\iGamma\phi\hx^0\dG\,.
\end{align}
This allows us to make use of the Poincar\'e-type inequality
\begin{align}\label{eq:poincare}
\norm{\trkla{\phi,\psi}}_{\widetilde{\mathcal{L}}^2}\leq C_P\norm{\nabla \phi}_{L^2\trkla{\Omega}}&&\text{for all~}\trkla{\phi,\psi}\in\mathcal{W}_{\beta,0}\,,
\end{align}
which was proven in \cite{GarckeKnopfYayla2022}.
For $\xi=0$, we also obtain the individual mass conservations
\begin{align}\label{eq:individualconservation}
\iOmega\phi\hx\nn\dx=\iOmega\phi\hx^0\dx&&\text{and}&&\iGamma\phi\hx\nn\dG=\iGamma\phi\hx^0\dG\,.
\end{align}
as we will show in the following, the presented scheme is unconditionally energy stable and has a unique solution. 
In particular, discrete solutions satisfy the following energy estimate.
\begin{lemma}\label{lem:energy}
Let the assumptions \ref{item:T}, \ref{item:S1}, \ref{item:S2}, and \ref{item:potential} hold true and let $\trkla{\phi\hx\no,r\hx\no,s\hx\no}$ be given.
Then, for $\xi<\infty$, a solution $\trkla{\phi\hx\nn,\mu\hx\nn,\theta\hx\nn,r\hx\nn,s\hx\nn}$ to \eqref{eq:modeldisc}, if it exists, satisfies
\begin{multline*}
\tfrac12\iOmega\tabs{\nabla\phi\hx\nn}^2\dx+\tfrac12\iGamma\tabs{\nablaG\phi\hx\nn}^2\dG+\tabs{r\hx\nn}^2+\tabs{s\hx\nn}^2 +\tfrac12\iOmega\tabs{\nabla\phi\hx\nn-\nabla\phi\hx\no}^2\dx\\
+\tfrac12\iGamma\tabs{\nablaG\phi\hx\nn-\nablaG\phi\hx\no}^2\dG + \tabs{r\hx\nn-r\hx\no}^2+\tabs{s\hx\nn-s\hx\no}^2\\
+\tau\iOmega\tabs{\nabla\mu\hx\nn}^2\dx+\tau\iGamma\tabs{\nablaG\theta\hx\nn}^2\dG+\xi\tau\iGamma\IhG{\tabs{\beta\theta\hx\nn-\mu\hx\nn}^2}\dG\\
=\tfrac12\iOmega\tabs{\nabla\phi\hx\no}^2\dx+\tfrac12\iGamma\tabs{\nablaG\phi\hx\no}^2\dG +\tabs{r\hx\no}^2+\tabs{s\hx\no}^2\,.
\end{multline*}
In the case $\xi=\infty$, a possible solution satisfies
\begin{multline*}
\tfrac12\iOmega\tabs{\nabla\phi\hx\nn}^2\dx+\tfrac12\iGamma\tabs{\nablaG\phi\hx\nn}^2\dG+\tabs{r\hx\nn}^2+\tabs{s\hx\nn}^2 +\tfrac12\iOmega\tabs{\nabla\phi\hx\nn-\nabla\phi\hx\no}^2\dx\\
+\tfrac12\iGamma\tabs{\nablaG\phi\hx\nn-\nablaG\phi\hx\no}^2\dG+ \tabs{r\hx\nn-r\hx\no}^2+\tabs{s\hx\nn-s\hx\no}^2\\
+\tau\iOmega\tabs{\nabla\mu\hx\nn}^2\dx+\tau\iGamma\tabs{\nablaG\theta\hx\nn}^2\dG\\
=\tfrac12\iOmega\tabs{\nabla\phi\hx\no}^2\dx+\tfrac12\iGamma\tabs{\nablaG\phi\hx\no}^2\dG +\tabs{r\hx\no}^2+\tabs{s\hx\no}^2\,,
\end{multline*}
together with $\beta\theta\hx\nn\equiv\mu\hx\nn$ on $\Gamma$.
\end{lemma}
\begin{remark}\label{rem:energy}
The result of Lemma \ref{lem:energy} indicates the drawback of the SAV approach.
Instead of controlling discrete versions of the energies defined in \eqref{eq:EOmega} and \eqref{eq:EGamma}, Lemma \ref{lem:energy} only allows us to control a modified energy functional, where $\iOmega\Ih{F\trkla{\phi\hx\nn}}\dx$ and $\iGamma\IhG{G\trkla{\phi\hx\nn}}\dG$ are replaced by $\tabs{r\hx\nn}^2$ and $\tabs{s\hx\nn}^2$.
\end{remark}
\begin{proof}[Proof of Lemma \ref{lem:energy}]
In the case $\xi<\infty$, we test \eqref{eq:modeldisc:phase1} by $\psi\h=\tau\mu\hx\nn$ and \eqref{eq:modeldisc:phase2} by $\eta\h=\trkla{1+\xi}\tau\trkla{\theta\hx\nn-\beta^{-1}\mu\hx\nn}$ to obtain
\begin{align}
\begin{split}\label{eq:energy:test1}
0=&\,\iOmega\Ih{\trkla{\phi\hx\nn-\phi\hx\no}\mu\hx\nn}\dx+\iGamma\IhG{\trkla{\phi\hx\nn-\phi\hx\no}\theta\hx\nn}\dG \\
&+\tau\iOmega\tabs{\nabla\mu\hx\nn}^2\dx+\tau\iGamma\tabs{\nablaG\theta\hx\nn}^2\dG +\xi\tau\iGamma\IhG{\tabs{\beta\theta\hx\nn-\mu\hx\nn}^2}\dG\,.
\end{split}
\end{align}
In the case $\xi=\infty$, \eqref{eq:modeldisc:phase2} immediately provides $\beta\theta\hx\nn\equiv\mu\hx\nn$ on $\Gamma$.
Therefore, it suffices to test \eqref{eq:modeldisc:phase1} by $\psi\h=\tau\mu\hx\nn$ to obtain
\begin{align}
\begin{split}\label{eq:energy:test2}
0=&\,\iOmega\Ih{\trkla{\phi\hx\nn-\phi\hx\no}\mu\hx\nn}\dx+\iGamma\IhG{\trkla{\phi\hx\nn-\phi\hx\no}\theta\hx\nn}\dG \\
&+\tau\iOmega\tabs{\nabla\mu\hx\nn}^2\dx+\tau\iGamma\tabs{\nablaG\theta\hx\nn}^2\dG\,.
\end{split}
\end{align}
Choosing $\hat{\psi}\h=\trkla{\phi\hx\nn-\phi\hx\no}$ in \eqref{eq:modeldisc:pot} and multiplying \eqref{eq:modeldisc:auxr} by $r\hx\nn$ and \eqref{eq:modeldisc:auxs} by $s\hx\nn$, we obtain
\begin{align}\label{eq:energy:test3}
\begin{split}
\iOmega&\Ih{\trkla{\phi\hx\nn-\phi\hx\no}\mu\hx\nn}\dx+\iGamma\IhG{\trkla{\phi\hx\nn-\phi\hx\no}\theta\hx\nn}\dG\\
&=\tfrac12\iOmega\tabs{\nabla\phi\hx\nn}^2\dx+\tfrac12\iOmega\tabs{\nabla\phi\hx\nn-\nabla\phi\hx\no}^2\dx-\tfrac12\iOmega\tabs{\nabla\phi\hx\no}^2\dx\\
&\quad+\tfrac12\iGamma\tabs{\nablaG\phi\hx\nn}^2\dG+\tfrac12\iGamma\tabs{\nablaG\phi\hx\nn-\nablaG\phi\hx\no}^2\dG-\tfrac12\iGamma\tabs{\nablaG\phi\hx\no}^2\dG\\
&\quad +\tabs{r\hx\nn}^2+\tabs{r\hx\nn-r\hx\no}^2-\tabs{r\hx\no}^2 +\tabs{s\hx\nn}^2+\tabs{s\hx\nn-s\hx\no}^2-\tabs{s\hx\no}^2\,.
\end{split}
\end{align}
Combining \eqref{eq:energy:test3} with \eqref{eq:energy:test1} or \eqref{eq:energy:test2}, respectively, provides the result.
\end{proof}

We shall now use these a priori estimates to establish the existence of a unique discrete solution.
Thereby, we will follow the approach used in \cite{Metzger2021a} and \cite{KnopfLamLiuMetzger2021}.
\begin{lemma}\label{lem:exitence}
Let the assumptions \ref{item:T}, \ref{item:S1}, \ref{item:S2}, and \ref{item:potential} hold true.
Then for given $\trkla{\phi\hx\no,r\hx\no,s\hx\no}\in\UhO\times\mathds{R}\times\mathds{R}$ and a given arbitrary time increment $\tau>0$, there exists a unique quintuple 
\begin{align*}
\trkla{\phi\hx\nn,\mu\hx\nn,\theta\hx\nn,r\hx\nn,s\hx\nn}\in \UhO\times\UhO\times\UhG\times\mathds{R}\times\mathds{R}
\end{align*}
which is a solution to \eqref{eq:modeldisc}.
\end{lemma}
\begin{proof}
Recalling \eqref{eq:conservation}, we note that the combined mean value of a possible solution $\phi\hx\nn$ is prescribed by $\phi\hx\no$.
Without loss of generality, we assume that $\iOmega\phi\hx\no\dx-\beta^{-1}\iGamma\phi\hx\no\dG=0$, which allows to apply \eqref{eq:poincare}.
As for fixed $\phi\hx\no$, $r\hx\no$, and $s\hx\no$ the quantities $r\hx\nn$ and $s\hx\nn$ are uniquely determined by $\phi\hx\nn$ via \eqref{eq:modeldisc:auxr} and \eqref{eq:modeldisc:auxs}, we may write \eqref{eq:modeldisc:pot} as
\begin{align}
\iOmega\Ih{\mu\hx\nn\psi\h}\dx+\iGamma\IhG{\theta\hx\nn\psi\h}\dG=\text{RHS}\trkla{\phi\hx\nn}\,.
\end{align}
This allows us to use the results from Section 5 in \cite{KnopfLamLiuMetzger2021} (see also \cite{Metzger2021a}) to derive the existence of potentials $\mu\hx\nn=:\mu\trkla{\phi\hx\nn}$ and $\theta\hx\nn=:\theta\trkla{\phi\hx\nn}$ satisfying \eqref{eq:modeldisc:pot} which are uniquely determined by $\phi\hx\nn$ such that \eqref{eq:modeldisc} equivalent to
\begin{multline}\label{eq:phireduced}
\iOmega\Ih{\trkla{\phi\hx\nn-\phi\hx\no}\psi\h}\dx+\beta^{-1}\iGamma\IhG{\trkla{\phi\hx\nn-\phi\hx\no}\psi\h}\dG\\
+\tau\iOmega\nabla\mu\trkla{\phi\hx\nn}\cdot\nabla\psi\h\dx+\beta^{-1}\tau\iGamma\nablaG\theta\trkla{\phi\hx\nn}\cdot\nablaG\psi\h\dG=0\,.
\end{multline}
In order to establish the existence of a suitable $\phi\hx\nn$ satisfying \eqref{eq:phireduced}, we shall use a fixed point argument based on Brouwer's fixed point theorem.\\

Assuming that \eqref{eq:phireduced} has no solution in the closed set
\begin{align}
B_R:=\gkla{\varphi\h\in\UhO\,: \iOmega\varphi\h\dx+\beta^{-1}\iGamma\varphi\h\dG=0\text{~and~}\norm{\varphi\h}_{L^2\trkla{\Omega}}\leq R}
\end{align}
for any $R>0$, the function $\mathcal{G}\,:\,\UhO\rightarrow\UhO$ defined via
\begin{multline}
\iOmega\Ih{\mathcal{G}\trkla{\phi\h}\psi\h}\dx+\beta^{-1}\iGamma\IhG{\mathcal{G}\trkla{\phi\h}\psi\h}\dG\\
:=\iOmega\Ih{\trkla{\phi\h-\phi\hx\no}\psi\h}\dx+\beta^{-1}\iGamma\IhG{\trkla{\phi\h-\phi\hx\no}\psi\h}\dG\\
+\tau\iOmega\nabla\mu\trkla{\phi\h}\cdot\nabla\psi\h\dx+\tau\beta^{-1}\iGamma\nablaG\theta\trkla{\phi\h}\cdot\nablaG\psi\h\dG
\end{multline}
has no roots in $B_R$.
Then, the function
\begin{align}\label{eq:def:H}
\mathcal{H}\trkla{\phi\h}:=-R\frac{\mathcal{G}\trkla{\phi\h}}{\norm{\mathcal{G}\trkla{\phi\h}}_{L^2\trkla{\Omega}}}
\end{align}
is a continuous mapping from $B_R$ to $\partial B_R\subset B_R$.
By Brouwer's fixed point theorem, we obtain the existence of at least one fixed point $\phi\h^*$ of $\mathcal{H}$.
In the following, we will show that this fixed point satisfies
\begin{align}\label{eq:contradiction}
0<\iOmega\Ih{\phi\h^*\mu\trkla{\phi\h^*}}\dx+\iGamma\IhG{\phi\h^*\theta\trkla{\phi\h^*}}\dG<0
\end{align}
for $R$ sufficiently large.
The arising contradiction shows that our initial assumption has to be wrong and there exists at least one solution to \eqref{eq:modeldisc}.
To prove the first inequality in \eqref{eq:contradiction}, we use that $\mu\trkla{\phi\h^*}$ and $\theta\trkla{\phi\h^*}$ satisfy \eqref{eq:modeldisc:pot} and compute
\begin{align}
\begin{split}
\iOmega&\Ih{\phi\h^*\mu\trkla{\phi\h^*}}\dx+\iGamma\IhG{\phi\h^*\theta\trkla{\phi\h^*}}\dG =\iOmega\tabs{\nabla\phi\h^*}^2\dx+\iGamma\tabs{\nablaG\phi\h^*}^2\dG \\
& +\frac{r\hx\no}{\sqrt{E\h^\Omega\trkla{\phi\hx\no}}}\iOmega\Ih{F^\prime\trkla{\phi\hx\no}\phi\h^*}\dx+\frac{1}{2 E\h^\Omega\trkla{\phi\hx\no}}\abs{\iOmega\Ih{F^\prime\trkla{\phi\hx\no}\phi\h^*}\dx}^2\\
&-\frac{1}{2E\h^\Omega\trkla{\phi\hx\no}}\iOmega\Ih{F^\prime\trkla{\phi\hx\no}\phi\hx\no}\dx\iOmega\Ih{F^\prime\trkla{\phi\hx\no}\phi\h^*}\dx  \\
&+\frac{s\hx\no}{\sqrt{E\h^\Gamma\trkla{\phi\hx\no}}}\iGamma\IhG{G^\prime\trkla{\phi\hx\no}\phi\h^*}\dG+\frac{1}{2 E\h^\Gamma\trkla{\phi\hx\no}}\abs{\iGamma\IhG{G^\prime\trkla{\phi\hx\no}\phi\h^*}\dG}^2\\
&-\frac{1}{2E\h^\Gamma\trkla{\phi\hx\no}}\iGamma\IhG{G^\prime\trkla{\phi\hx\no}\phi\hx\no}\dG\iGamma\IhG{G^\prime\trkla{\phi\hx\no}\phi\h^*}\dG\,.
\end{split}
\end{align}
By applying Young's inequality and \eqref{eq:poincare}, we obtain
\begin{align}
\iOmega\Ih{\phi\h^*\mu\trkla{\phi\h^*}}\dx+\iGamma\IhG{\phi\h^*\theta\trkla{\phi\h^*}}\dG\geq  C_P^{-2}\norm{\trkla{\phi\h^*,\phi\h^*\big\vert_\Gamma}}_{\widetilde{\mathcal{L}}^2}^2 -C\trkla{\phi\hx\no,r\hx\no,s\hx\no}
\end{align}
which is positive for $R$ large enough.\\
The second inequality in \eqref{eq:contradiction} can be established by recalling the computations from the proof of Lemma \ref{lem:energy} and \eqref{eq:poincare} which provide
\begin{align*}
\iOmega\!\Ih{\mathcal{G}\trkla{\phi\h^*}\mu\trkla{\phi\h^*}}\dx\!+\!\!\iGamma\!\IhG{\mathcal{G}\trkla{\phi\h^*}\theta\trkla{\phi\h^*}}\dG \geq \tfrac{C_P^{-2}}2\norm{\trkla{\phi\h^*,\phi\h^*\big\vert_\Gamma}}_{\widetilde{\mathcal{L}}^2}^2 -C\trkla{\phi\hx\no,r\hx\no,s\hx\no}\,.
\end{align*}
Again, we may choose $R$ sufficiently large to obtain a positive right-hand side.
Hence, by \eqref{eq:def:H} we have the second inequality in \eqref{eq:contradiction} and therefore the existence of at least one solution.\\
To obtain the uniqueness, we assume the existence of two solutions denoted by $\trkla{\phi\hx\nn,\mu\hx\nn,\theta\hx\nn,r\hx\nn,s\hx\nn}$ and $\trkla{\tilde{\phi}\hx\nn,\tilde{\mu}\hx\nn,\tilde{\theta}\hx\nn,\tilde{r}\hx\nn,\tilde{s}\hx\nn}$, use the linearity of \eqref{eq:modeldisc} to substract the corresponding equations, and mimic the steps from the proof of Lemma \ref{lem:energy} to obtain
\begin{multline}\label{eq:soldifference}
\norm{\nabla\phi\hx\nn-\nabla\tilde{\phi}\hx\nn}_{L^2\trkla{\Omega}}^2+\norm{\nablaG\phi\hx\nn-\nablaG\tilde{\phi}\hx\nn}_{L^2\trkla{\Gamma}}^2 +\tabs{r\hx\nn-\tilde{r}\hx\nn}^2 +\tabs{s\hx\nn-\tilde{s}\hx\nn}^2\\
+\tau\norm{\nabla\mu\hx\nn-\nabla\tilde{\mu}\hx\nn}_{L^2\trkla{\Omega}}^2 +\tau\norm{\nablaG\theta\hx\nn-\nablaG\tilde{\theta}\hx\nn}_{L^2\trkla{\Gamma}}^2\leq 0\,.
%+\tau\iOmega\tabs{\nabla\trkla{\mu\hx\nn-\tilde{\mu}\hx\nn}}^2\dx+\tau\iGamma\tabs{\nablaG\trkla{\theta\hx\nn-\tilde{\theta}\hx\nn}}^2\dG\leq 0\,.
\end{multline}
Recalling \eqref{eq:poincare}, we note $\phi\hx\nn=\tilde{\phi}\hx\nn$, $r\hx\nn=\tilde{r}\hx\nn$, and $s\hx\nn=\tilde{s}\hx\nn$.
Concerning the chemical potentials, \eqref{eq:soldifference} suggests uniqueness up to an additive constant.
However, recalling that $\mu\hx\nn$ and $\theta\hx\nn$ are uniquely determined by $\phi\hx\nn$, we are able to resolve this ambiguity.
\end{proof}
\section{Compactness in space and time}\label{sec:regularity}
In this section, we derive improved regularity results which will enable us to pass to the limit $\trkla{h,\tau,\xi}\rightarrow\trkla{0,0,\xi_0}$ for $\xi_0\in\tekla{0,\infty}$ in Section \ref{sec:limit} by identifying weakly and strongly converging subsequences of discrete solutions.
Summing over all time steps, a direct consequence of Lemma \ref{lem:energy} is the following result:
\begin{lemma}\label{lem:regularity}
Let the assumptions \ref{item:T}, \ref{item:S1}, \ref{item:S2}, \ref{item:potential}, \ref{item:htau1}, and \ref{item:I} hold true and let $h>0$ be small enough. Then the solution $\trkla{\phi\hx\nn,\,\mu\hx\nn,\theta\hx\nn,r\hx\nn,s\hx\nn}_{n=1,\ldots,N}$ to \eqref{eq:modeldisc} satisfies
\begin{subequations}\label{eq:regularity:energy}
\begin{multline}\label{eq:regularity:energy:a}
\max_{n=0,\ldots,N}\norm{\phi\hx\nn}_{H^1\trkla{\Omega}}^2+\max_{n=0,\ldots,N}\norm{\phi\hx\nn}_{H^1\trkla{\Gamma}}^2 + \max_{n=0,\ldots,N} \tabs{r\hx\nn}^2 +\max_{n=0,\ldots,N}\tabs{s\hx\nn}^2\\
+\sum_{n=1}^N\norm{\phi\hx\nn-\phi\hx\no}_{H^1\trkla{\Omega}}^2+\sum_{n=1}^N\norm{\phi\hx\nn-\phi\hx\no}_{H^1\trkla{\Gamma}}^2 +\sum_{n=1}^N\tabs{r\hx\nn-r\hx\no}^2 +\sum_{n=1}^N\tabs{s\hx\nn-s\hx\no}^2\\
+\tau\sum_{n=1}^N\norm{\mu\hx\nn}_{H^1\trkla{\Omega}}^2+\tau\sum_{n=1}^N\norm{\theta\hx\nn}_{H^1\trkla{\Gamma}}^2\leq C
\end{multline}
with a constant $C>0$ depending only on $\phi_0$, but not on $h$, $\tau$, or $\xi$.
Additionally, we have
\begin{align}
\tau\sum_{n=1}^N\norm{\beta\theta\hx\nn-\mu\hx\nn}_{L^2\trkla{\Gamma}}^2\leq C\xi^{-1}\,.
\end{align}
\end{subequations}
\end{lemma}
\begin{proof}
Summing the results of Lemma \ref{lem:energy} over all time steps, recalling the norm equivalence \eqref{eq:normequivalence} and the Poincar\'e-type estimate \eqref{eq:poincare}, and applying \eqref{eq:intialbounds}, we obtain \eqref{eq:regularity:energy} with bounds on the $H^1$-semi-norms of $\mu\hx\nn$ and $\theta\h\nn$.
To establish bounds for their complete $H^1$-norms, we extend the ideas of Lemma 4.5 in \cite{GarckeKnopf20}.\\
Testing \eqref{eq:modeldisc:pot} by $\Ih{\psi}$ with $\psi\in C_0^\infty\trkla{\Omega;\tekla{0,1}}$, which is not identically zero, we obtain
\begin{align}
\begin{split}
\iOmega\Ih{\mu\hx\nn\psi}\dx=&\,\iOmega\nabla\phi\hx\nn\cdot\nabla\Ih{\psi}\dx + \frac{r\hx\nn}{\sqrt{E\h^\Omega\trkla{\phi\hx\no}}}\iOmega\Ih{F^\prime\trkla{\phi\hx\no}\psi}\dx\,.
\end{split}
\end{align}
From \ref{item:potential}, the norm equivalence \eqref{eq:normequivalence}, and the already established regularity results, we obtain
\begin{align}\label{eq:est:potential}
\abs{\frac{r\hx\nn}{\sqrt{E\h^\Omega\trkla{\phi\hx\no}}}\iOmega\Ih{F^\prime\trkla{\phi\hx\no}\psi}\dx}\leq C\gamma^{-1/2}\norm{\psi}_{L^\infty\trkla{\Omega}}\tabs{r\hx\nn}\rkla{\norm{\phi\hx\no}_{L^3\trkla{\Omega}}^3+C}\leq C\trkla{\psi}\,.
\end{align}
Hence, there exists a constant $\tilde{C}_1\trkla{\psi}>0$ independent of $h$, $\tau$, and $\xi$ such that $\displaystyle \abs{\iOmega\Ih{\mu\hx\nn\psi}\dx}\leq\tilde{C}_1\trkla{\psi}$.

We shall now derive a similar estimate without the interpolation operator.
Denoting the mean value of $\mu\hx\nn$ by $m\hx\nn$, we compute
\begin{align}
\begin{split}
\abs{\iOmega\mu\hx\nn\psi\dx}\leq& \abs{\iOmega\Ih{\mu\hx\nn\psi}\dx}+\abs{\iOmega\trkla{1-\Ihop}\gkla{\mu\hx\nn\Ih{\psi}}\dx}\\
&+\abs{\iOmega\overline{\mu}\hx\nn\trkla{1-\Ihop}\tgkla{\psi}\dx}+ \abs{m\hx\nn\iOmega\trkla{1-\Ihop}\tgkla{\psi}\dx}\\
=:&\, I+II+III+IV\,,
\end{split}
\end{align}
where $\overline{\mu}\hx\nn:=\mu\hx\nn-m\hx\nn$ has mean value zero.
The first term $I$ is bounded by $\tilde{C}_1\trkla{\psi}$.
Applying Lemma \ref{lem:Ih} and \ref{item:htau1}, we note that $II$ is bounded by $\tilde{C}_2\trkla{\psi}\geq C\trkla{\psi}\tfrac{h^2}{\sqrt{\tau}}\sqrt{\tau}\norm{\nabla\mu\hx\nn}_{L^2\trkla{\Omega}}$.
Combining Hölder's inequality, Poincar\'e's inequality, standard error estimates for the interpolation operator, and assumption \ref{item:htau1}, we obtain
\begin{align}
III\leq \norm{\overline{\mu}\hx\nn}_{L^2\trkla{\Omega}}\norm{\trkla{1-\Ihop}\tgkla{\psi}}_{L^2\trkla{\Omega}}\leq C\trkla{\psi}\sqrt{\tau}\norm{\nabla\mu\hx\nn}_{L^2\trkla{\Omega}}\leq \tilde{C}_3\trkla{\psi}\,.
\end{align}
To tackle $IV$, we use the standard error estimates for the interpolation operator $\Ihop$ to derive the existence of $c\trkla{\psi}> 0$ such that $\displaystyle\iOmega\Ih{\psi} \dx\geq c\trkla{\psi}$ for $h$ small enough.
This allows us to compute
\begin{align}
\begin{split}
III\leq&\,C\tabs{m\hx\nn} h^2\norm{\psi}_{H^2\trkla{\Omega}}\leq C\trkla{\psi}\tabs{m\hx\nn} h^2\iOmega\Ih{\psi}\dx\\
\leq&\,C\trkla{\psi}h^2\abs{\iOmega\trkla{m\hx\nn+\overline{\mu}\hx\nn}\Ih{\psi}\dx} + C\trkla{\psi}h^2\abs{\iOmega\overline{\mu}\hx\nn\Ih{\psi}\dx}\\
\leq&\,C\trkla{\psi} h^2 \trkla{I+II} + C\trkla{\psi}\sqrt{\tau}\norm{\nabla \mu\hx\nn}_{L^2\trkla{\Omega}}\norm{\Ih{\psi}}_{L^2\trkla{\Omega}}\\
\leq&\, C\trkla{\psi} \tilde{C}_1\trkla{\psi} + C\trkla{\psi}:=\tilde{C}_4\trkla{\psi}\,,
\end{split}
\end{align}
where we used \ref{item:htau1}, Poincar\'e's inequality, and the previously obtained regularity results.
Therefore, for $h$ small enough, $\mu\hx\nn$ is in the set
\begin{align}
\mathcal{M}_\psi:=\gkla{v\in H^1\trkla{\Omega}\,:\,\abs{\iOmega v\psi\dx}\leq \hat{C}\trkla{\psi}:=\tilde{C}_1\trkla{\psi}+\tilde{C}_2\trkla{\psi}+\tilde{C}_3\trkla{\psi}+\tilde{C}_4\trkla{\psi}}\,
\end{align}
with a constant $\hat{C}\trkla{\psi}$ depending on $\psi$ but not on $h$, $\tau$, or $\xi$.
As the set $\mathcal{M}_\psi$ is obviously a non-empty, closed, and convex subset of $H^1\trkla{\Omega}$, we can apply the generalized Poincar\'e's inequality (cf.~page 242 in \cite{Alt2016}) to obtain
\begin{align}
\norm{\mu\hx\nn}_{L^2}\leq C\trkla{1+\norm{\nabla\mu\hx\nn}_{L^2}}
\end{align}
with $C>0$ independent of $h$, $\tau$, or $\xi$, which provides the $H^1$-bound on $\mu\hx\nn$.
Testing \eqref{eq:modeldisc:pot} by $\psi\h\equiv1$ provides
\begin{multline}
\abs{\iGamma \theta\hx\nn\dG}\leq \abs{\iOmega\mu\hx\nn\dx} +\abs{\frac{r\hx\nn}{\sqrt{E\h^\Omega\trkla{\phi\hx\no}}}\iOmega\Ih{F^\prime\trkla{\phi\hx\no}}\dx}\\
 +\abs{\frac{s\hx\nn}{\sqrt{E\h^\Gamma\trkla{\phi\hx\no}}}\iGamma \IhG{G^\prime\trkla{\phi\hx\no}}\dG}\,.
\end{multline}
Then, considerations similar to \eqref{eq:est:potential} and the already established $L^\infty\trkla{H^1}$-bounds on $\phi\hx\nn$ indicate that the right-hand side is bounded independently of $h$, $\tau$, and $\xi$.
Therefore, applying Poincar\'e's inequality completes the proof.
\end{proof}

In the continuous setting, it can be proven that $\para{t}\phi\in L^2\trkla{0,T;\mathcal{V}^\prime}$ (see also Rem.~\ref{rem:uniqueness}).
In the discrete setting, however, such an estimate requires a $\mathcal{V}$-stable projection operator mapping $\mathcal{V}$ onto $\UhO\times\UhG$.
We shall therefore follow a different pathway and establish the following Nikolskii-type estimates:
\begin{lemma}\label{lem:nikolskii:phi}
Suppose that the assumptions \ref{item:T}, \ref{item:S1}, \ref{item:S2}, \ref{item:potential}, \ref{item:htau1}, and \ref{item:I} hold true. 
Then the solution $\trkla{\phi\hx\nn}_{n=1,\ldots,N}$ to \eqref{eq:modeldisc} satisfies 
\begin{subequations}
\begin{align}\label{eq:nikoslkii:phi}
\tau\sum_{k=0}^{N-l}\norm{\phi\hx^{k+l}-\phi\hx^k}_{L^2\trkla{\Omega}}^2+\tau\sum_{k=0}^{N-l}\norm{\phi\hx^{k+l}-\phi\hx^k}_{L^2\trkla{\Gamma}}^2\leq C\tau l\,,
\end{align}
for $l\in\tgkla{1,\ldots,N}$ with $C>0$ independent of $l$, $\xi$, $h$, and $\tau$.
Furthermore, we have
\begin{align}\label{eq:improvedphi}
\tau\sum_{k=0}^{N-1}\norm{\phi\hx^{k+1}-\phi\hx^k}_{L^2\trkla{\Omega}}^2+\tau\sum_{k=0}^{N-1}\norm{\phi\hx^{k+1}-\phi\hx^k}_{L^2\trkla{\Gamma}}^2\leq C\tau^{3/2}
\end{align}
\end{subequations}
with $C>0$ independent of $\xi$, $h$, and $\tau$.

\end{lemma}
\begin{proof}
For $0\leq k\leq N-l$, we test \eqref{eq:modeldisc:phase1} by $\psi\h=\trkla{\phi\hx^{k+l}-\phi\hx^k}$ and sum from $n=k+1$ to $k+l$ to obtain
\begin{align}
\begin{split}\label{eq:nikoslkii:phi:1}
\iOmega&\Ih{\tabs{\phi\hx^{k+l}-\phi\hx^k}^2}\dx+\beta^{-1}\iGamma\IhG{\tabs{\phi\hx^{k+l}-\phi\hx^k}^2}\dG\\
&\leq \abs{\tau\sum_{n=k+1}^{k+l}\iOmega\nabla\mu\hx\nn\cdot\nabla\trkla{\phi\hx^{k+l}-\phi\hx^k}\dx} +\beta^{-1}\abs{\tau\sum_{n=k+1}^{k+l}\iGamma\nablaG\theta\hx\nn\cdot\nablaG\trkla{\phi\hx^{k+l}-\phi\hx^k}\dG}\\
&\leq \tau\sum_{n=k+1}^{k+l}\norm{\nabla\mu\hx\nn}_{L^2\trkla{\Omega}}\norm{\nabla\phi\hx^{k+l}-\nabla \phi\hx^k}_{L^2\trkla{\Omega}}\\
&\qquad+\beta^{-1}\tau\sum_{n=k+1}^{k+l} \norm{\nablaG\theta\hx\nn}_{L^2\trkla{\Gamma}}\norm{\nablaG\phi\hx^{k+l}-\nablaG\phi\hx^k}_{L^2\trkla{\Gamma}}\,.
\end{split}
\end{align}
In order to prove \eqref{eq:nikoslkii:phi}, we multiply \eqref{eq:nikoslkii:phi:1} by $\tau$ and sum from $k=0$ to $N-l$.
Using the norm equivalence \eqref{eq:normequivalence} and the regularity results from Lemma \ref{lem:regularity}, we obtain
\begin{align}
\begin{split}
\tau&\sum_{k=0}^{N-l}\norm{\phi\hx^{k+l}-\phi\hx^k}_{L^2\trkla{\Omega}}^2+\tau\sum_{k=0}^{N-l}\norm{\phi\hx^{k+l}-\phi\hx^k}_{L^2\trkla{\Gamma}}^2\\
&\leq C\tau^2\sum_{m=1}^{l}\ekla{\rkla{\sum_{k=0}^{N-l}\norm{\nabla\phi\hx^{k+l}-\nabla\phi\hx^k}_{L^2\trkla{\Omega}}^2}^{1/2}\rkla{\sum_{k=0}^{N-l}\norm{\nabla\mu\hx^{k+m}}_{L^2\trkla{\Omega}}}^{1/2}}\\
&\qquad +C\tau^2\sum_{m=1}^l\ekla{\rkla{\sum_{k=0}^{N-l}\norm{\nablaG\phi\hx^{k+l}-\nablaG\phi\hx^k}_{L^2\trkla{\Gamma}}^2}^{1/2}\rkla{\sum_{k=0}^{N-l}\norm{\nablaG\theta\hx^{k+m}}_{L^2\trkla{\Gamma}}^2}^{1/2}}\\
&\leq C\tau l\rkla{\tau N\max_{n=1,\ldots,N}\norm{\nabla\phi\hx\nn}_{L^2\trkla{\Omega}}^2}^{1/2}\rkla{\tau\sum_{n=1}^N\norm{\nabla\mu\hx\nn}_{L^2\trkla{\Omega}}^2}^{1/2}\\
&\qquad +C\tau l\rkla{\tau N\max_{n=1,\ldots,N}\norm{\nablaG\phi\hx\nn}_{L^2\trkla{\Gamma}}^2}^{1/2}\rkla{\tau\sum_{n=1}^N\norm{\nablaG\theta\h\nn}_{L^2\trkla{\Gamma}}^2}^{1/2}\leq C\tau l\,.
\end{split}
\end{align}
In the special case $l=1$, we are able to obtain better estimates for the right-hand side of \eqref{eq:nikoslkii:phi:1}.
In particular, combining the norm equivalence \eqref{eq:normequivalence} and the regularity results from Lemma \ref{lem:regularity}, provides \eqref{eq:improvedphi} via
\begin{align}
\begin{split}
\tau&\sum_{k=0}^{N-1}\norm{\phi\hx^{k+1}-\phi\hx^k}_{L^2\trkla{\Omega}}^2+\tau\sum_{k=0}^{N-1}\norm{\phi\hx^{k+1}-\phi\hx^k}_{L^2\trkla{\Gamma}}^2\\
&\leq C\tau^{3/2}\rkla{\sum_{k=0}^{N-1}\norm{\nabla\phi\hx^{k+1}-\nabla\phi\hx^k}_{L^2\trkla{\Omega}}^2}^{1/2}\rkla{\tau\sum_{n=1}^N\norm{\nabla\mu\hx\nn}_{L^2\trkla{\Omega}}^2}^{1/2}\\
&\qquad +C\tau^{3/2}\rkla{\sum_{k=0}^{N-1}\norm{\nablaG\phi\hx^{k+1}-\nablaG\phi\hx^k}_{L^2\trkla{\Gamma}}^2}^{1/2}\rkla{\tau\sum_{n=1}^N\norm{\nablaG\theta\hx\nn}_{L^2\trkla{\Gamma}}^2}^{1/2}\leq C\tau^{3/2}\,.
\end{split}
\end{align}
\end{proof}
\begin{remark}
Comparing the estimate \eqref{eq:improvedphi} with the estimates on $\nabla\phi\hx\nn-\nabla\phi\hx\no$ established in Lemma \ref{lem:regularity}, we obtain by applying inverse estimates
\begin{multline}
\sum_{n=1}^N\norm{\phi\hx\nn-\phi\hx\no}_{H^1\trkla{\Omega}}^2+\sum_{n=1}^N\norm{\phi\hx\nn-\phi\hx\no}_{H^1\trkla{\Gamma}}^2\\
\leq C h^{-2}\rkla{\sum_{n=1}^N\norm{\phi\hx\nn-\phi\hx\no}_{L^2\trkla{\Omega}}^2 + \sum_{n=1}^N\norm{\phi\hx\nn-\phi\hx\no}_{L^2\trkla{\Gamma}}^2}\leq C\frac{\sqrt{\tau}}{h^2}\,.
\end{multline}
Recalling assumption \ref{item:htau1}, we note that $\displaystyle \frac{\sqrt{\tau}}{h^2}\nearrow \infty$.
Therefore, Lemma \ref{lem:nikolskii:phi} improves the bounds on the $L^2$-norms, but is insufficient to enhance the bounds on the $H^1$-norms.
\end{remark}

\section{Convergence to the limit}\label{sec:limit}
We will now use the regularity results establish in the last section to identify converging subsequences and pass to the limit $\trkla{h,\tau,\xi}\rightarrow\trkla{0,0,\xi_0}$ in \eqref{eq:modeldisc}.
Therefore, we define time-interpolants of time-discrete functions $a\nn$, $n=0,\ldots,N$, and introduce some time-index-free notation as follows:
\begin{subequations}\label{eq:def:interpolation}
\begin{align}
a\tl\trkla{.,t}&:=\tfrac{t-t\no}{\tau}a\nn\trkla{.}+\tfrac{t\nn-t}{\tau}a\no\trkla{.}&t&\in\tekla{t\no,t\nn},~n\geq1\,\,,\\
a\tp\trkla{.,t}&:=a\nn\trkla{.}\,,\quad a\tm\trkla{.,t}:=a\no\trkla{.}&t&\in(t\no,t\nn]\,,~n\geq 1\,.
\end{align}
\end{subequations}
We want to point out that the time derivative of the piecewise linear interpolation $a\tl$ coincides with the difference quotient in time, i.e.~$\para{t}a\tl=\dtau a\nn$.
%\begin{align}
%\para{t}a\tl=\dtau a\nn\,.
%\end{align}
If a statement is valid for all three time-interpolants defined in \eqref{eq:def:interpolation}, we shall use the abbreviation $a\tpm$.\\
Using the interpolants defined in \eqref{eq:def:interpolation}, we are able to summarize the regularity results obtained in the last section as follows
\begin{subequations}\label{eq:regularity}
\begin{multline}\label{eq:regularity:energy}
%\begin{split}
\norm{\phi\hx\tpm}_{L^\infty\trkla{0,T;H^1\trkla{\Omega}}}^2+\norm{\phi\hx\tpm}_{L^\infty\trkla{0,T;H^1\trkla{\Gamma}}}^2+\norm{r\hx\tpm}_{L^\infty\trkla{0,T}}^2 +\norm{s\hx\tpm}_{L^\infty\trkla{0,T}}^2\\
+\tau^{-1}\norm{\phi\hx\tp-\phi\hx\tm}_{L^2\trkla{0,T;H^1{\trkla{\Omega}}}}^2+\tau^{-3/2}\norm{\phi\hx\tp-\phi\hx\tm}_{L^2\trkla{0,T;L^2\trkla{\Omega}}}^2\\
+\tau^{-1}\norm{\phi\hx\tp-\phi\hx\tm}_{L^2\trkla{0,T;H^1{\trkla{\Gamma}}}}^2+\tau^{-3/2}\norm{\phi\hx\tp-\phi\hx\tm}_{L^2\trkla{0,T;L^2\trkla{\Gamma}}}^2\\
+\tau^{-1}\norm{r\hx\tp-r\hx\tm}_{L^2\trkla{0,T}}^2 +\tau^{-1}\norm{s\hx\tp-s\hx\tm}_{L^2\trkla{0,T}}^2 \\
+\norm{\mu\hx\tp}_{L^2\trkla{0,T;H^1\trkla{\Omega}}}^2+\norm{\theta\hx\tp}_{L^2\trkla{0,T;H^1\trkla{\Gamma}}}^2\leq C\,,
%\end{split}
\end{multline}
\begin{align}
\xi\norm{\beta\theta\hx\tp-\mu\hx\tp}_{L^2\trkla{0,T;L^2\trkla{\Gamma}}}^2&\leq C\,,\label{eq:regularity:jump}
\end{align}
\begin{align}
\norm{\phi\hx\tpm\trkla{\cdot+ l\tau}-\phi\hx\tpm}_{L^2\trkla{0,T-l\tau;L^2\trkla{\Omega}}}^2\!\!+\norm{\phi\hx\tpm\trkla{\cdot+ l\tau}-\phi\hx\tpm}_{L^2\trkla{0,T-l\tau;L^2\trkla{\Gamma}}}^2\!&\leq Cl\tau\,.\label{eq:regularity:nikolskii}
\end{align}
\end{subequations}

\begin{lemma}\label{lem:convergence}
Let the assumptions \ref{item:T}, \ref{item:S1}, \ref{item:S2}, \ref{item:potential}, \ref{item:htau1}, and \ref{item:I} hold true.
Then, there exist functions $\trkla{\phi,\,\phi_\Gamma,\,\mu,\,\theta,\,r,\,s}$ and a subsequence, which is again denoted by $\gkla{\phi\hx\tpm,\,\mu\hx\tp,\,\theta\hx\tp,\, r\hx\tpm,\, s\hx\tpm}_{h,\tau,\xi}$ satisfying
\begin{align}
\left\{\begin{array}{ll}
\phi\in L^\infty\trkla{0,T;H^1\trkla{\Omega}}\,,&\phi_{\Gamma}\in L^\infty\trkla{0,T;H^1\trkla{\Gamma}}\,,\\
\mu\in L^2\trkla{0,T;H^1\trkla{\Omega}}\,,&\theta\in L^2\trkla{0,T;H^1\trkla{\Gamma}}\,,\\
r\in L^\infty\trkla{0,T}\,,&s\in L^\infty\trkla{0,T}
\end{array}\right.
\end{align}
such that $\phi_\Gamma=\phi\big\vert_{\Gamma}$ a.e.~on $\Gamma\times\trkla{0,T}$, and
\begin{subequations}
\begin{align}
\phi\hx\tpm&\weakstar\phi&&\text{in~}L^\infty\trkla{0,T;H^1\trkla{\Omega}}\,,\label{eq:conv:phiO:weakstar}\\
\phi\hx\tpm &\rightarrow \phi&&\text{in~} L^r\trkla{0,T;L^s\trkla{\Omega}}\text{~with~}r<\infty,\,s\in[1,\tfrac{2d}{d-2})\,,\label{eq:conv:phiO:strong}\\
\phi\hx\tpm\big\vert_{\Gamma}&\weakstar \phi_\Gamma&&\text{in~}L^\infty\trkla{0,T;H^1\trkla{\Gamma}}\,,\label{eq:conv:phiG:weakstar}\\
\phi\hx\tpm\big\vert_{\Gamma}&\rightarrow \phi_\Gamma&&\text{in~}L^r\trkla{0,T;L^s\trkla{\Gamma}}\text{~with~}r,\,s<\infty\,,\label{eq:conv:phiG:strong}\\
\mu\hx\tp&\weak\mu&&\text{in~} L^2\trkla{0,T;H^1\trkla{\Omega}}\,,\label{eq:conv:muO:weak}\\
\mu\hx\tp\big\vert_{\Gamma}&\weak\mu\big\vert_{\Gamma}&&\text{in~} L^2\trkla{0,T;H^{1/2}\trkla{\Gamma}}\,,\label{eq:conv:muOtrace:weak}\\
\theta\hx\tp&\weak\theta&&\text{in~} L^2\trkla{0,T;H^1\trkla{\Gamma}}\,,\label{eq:conv:theta:weak}\\
r\hx\tpm&\weakstar r&&\text{in~}L^\infty\trkla{0,T}\,,\label{eq:conv:r:weakstar}\\
s\hx\tpm&\weakstar s&&\text{in~}L^\infty\trkla{0,T}\,\label{eq:conv:s:weakstar}
\end{align}
as $\trkla{h,\tau,\xi}\rightarrow\trkla{0,0,\xi_0}$.
If $\xi_0=\infty$, we additionally obtain
\begin{align}
\beta\theta\hx\tp-\mu\hx\tp\big\vert_{\Gamma}\rightarrow 0&&\text{in~}L^2\trkla{0,T;L^2\trkla{\Gamma}}\label{eq:conv:jump}\,.
\end{align}
\end{subequations}
\end{lemma}
\begin{proof}
The convergence results expressed in \eqref{eq:conv:phiO:weakstar}, \eqref{eq:conv:phiG:weakstar}, \eqref{eq:conv:muO:weak}, \eqref{eq:conv:theta:weak}, \eqref{eq:conv:r:weakstar}, and \eqref{eq:conv:s:weakstar} are direct consequences of the bounds listed in \eqref{eq:regularity:energy}.

The strong convergence of $\phi\hx\tpm$ stated in \eqref{eq:conv:phiO:strong} and \eqref{eq:conv:phiG:strong} follows from Simon's compactness result (cf.~Th.~5 in \cite{Simon1987}) in combination with the bounds in \eqref{eq:regularity:energy} and \eqref{eq:regularity:nikolskii}.\\
Due to the trace theorem, the uniform bound on $\mu\hx\tp$ in $L^2\trkla{0,T;H^1\trkla{\Omega}}$ also provides a uniform bound in $L^2\trkla{0,T;H^{1/2}\trkla{\Gamma}}$.
Therefore, there exists a subsequence -- again denoted by $\gkla{\mu\hx\tp\big\vert_{\Gamma}}_{h,\tau,\xi}$ -- converging towards a limit function $\nu\in L^2\trkla{0,T; H^{1/2}\trkla{\Gamma}}$.
To obtain \eqref{eq:conv:muOtrace:weak}, we need to identify this limit function $\nu$ with $\mu\big\vert_{\Gamma}$.
Following the lines of \cite{Metzger2021a}, we choose $\bs{\psi}\in L^2\trkla{0,T;\trkla{C^\infty\trkla{\overline{\Omega}}}^d}$ and compute
\begin{multline}
\iOmegaT \mu\div\bs{\psi}\dx\dt\leftarrow \iOmegaT\mu\hx\tp\div\bs{\psi}\dx\dt\\
=-\iOmegaT\nabla\mu\hx\tp\cdot\bs{\psi}\dx\dt+\iGammaT \mu\hx\tp\big\vert_{\Gamma}\bs{\psi}\cdot\bs{n}\dG\dt
\rightarrow -\iOmegaT\nabla\mu\cdot\bs{\psi}\dx\dt+\iGammaT\nu\bs{\psi}\cdot\bs{n}\dG\dt\\
=\iOmegaT\mu\div\bs{\psi}\dx\dt+\iGammaT\rkla{\nu-\mu\big\vert_{\Gamma}}\bs{\psi}\cdot\bs{n}\dG\dt\,.
\end{multline}
The identification of $\phi_\Gamma$ with $\phi\big\vert_{\Gamma}$ follows by a similar argument.
Finally, \eqref{eq:conv:jump} is a consequence of \eqref{eq:regularity:jump}.
\end{proof}
\begin{corollary}\label{cor:convergenceE}
Let the assumptions \ref{item:T}, \ref{item:S1}, \ref{item:S2}, \ref{item:potential}, \ref{item:htau1}, and \ref{item:I} hold true.
Then, there exists a subsequence $\tgkla{\phi\hx\tpm}_{h,\tau,\xi}$ such that 
\begin{subequations}
\begin{align}
\iOmega \Ih{ F\trkla{\phi\hx\tpm}}\dx&\rightarrow \iOmega F\trkla{\phi}\dx&&\text{in~} L^p\trkla{0,T}\,,\label{eq:convEO}\\
\iGamma \IhG{ G\trkla{\phi\hx\tpm}}\dG&\rightarrow \iGamma G\trkla{\phi\big\vert_{\Gamma}}\dG&&\text{in~} L^p\trkla{0,T}\label{eq:convEG}\,,
\end{align}
\end{subequations}
for $p<\infty$ as $\trkla{h,\tau,\xi}\rightarrow\trkla{0,0,\xi_0}$.
\end{corollary}
\begin{proof}
To establish \eqref{eq:convEO}, we shall first show that the nodal interpolation is negligible when passing to the limit.
Looking at the error on each $K\in\Th$ individually, we note that $\nabla\phi\hx\tpm$ is constant on each $K$ and apply \ref{item:potential} and an inverse estimate (cf.~Lemma 4.5.3 in \cite{BrennerScott}) to obtain
\begin{align}
\begin{split}
\int_{K}&\abs{\trkla{1-\Ihop}\gkla{F\trkla{\phi\hx\tpm}}}\dx \leq \tabs{K}\norm{\trkla{1-\Ihop}\gkla{F\trkla{\phi\hx\tpm}}}_{L^\infty\trkla{K}}\\
&\quad\leq C\tabs{K}h\norm{\nabla F\trkla{\phi\hx\tpm}}_{L^\infty\trkla{K}}=C\tabs{K}h\norm{F^\prime\trkla{\phi\hx\tpm}\nabla\phi\hx\tpm}_{L^\infty\trkla{K}}\\
&\quad\leq C\tabs{K} h \rkla{1+\norm{\phi\hx\tpm}_{L^\infty\trkla{K}}^3}\abs{\nabla\phi\hx\tpm}\\
&\quad\leq C\tabs{K} h\rkla{1+\norm{\phi\hx\tpm}_{L^\infty\trkla{K}}^6+ \abs{\nabla\phi\hx\tpm}^2}\\
&\quad\leq C\tabs{K} h + C\tabs{K} h h^{-d}\norm{\phi\hx\tpm}_{L^6\trkla{K}}^6 + C h \norm{\nabla\phi\hx\tpm}_{L^2\trkla{K}}^2\,.
\end{split}
\end{align}
Due to \ref{item:S1}, we have $\tabs{K}\leq C h^d$. Therefore, summing over all $K\in\Th$ and recalling \eqref{eq:regularity:energy:a} yields
\begin{multline}
\esssup_{t\in\trkla{0,T}}\iOmega\abs{\trkla{1-\Ihop}\gkla{F\trkla{\phi\hx\tpm}}}\dx\\
\leq\, Ch\tabs{\Omega} + Ch\esssup_{t\in\trkla{0,T}}\norm{\phi\hx\tpm}_{H^1\trkla{\Omega}}^6 + Ch\esssup_{t\in\trkla{0,T}}\norm{\nabla\phi\hx\tpm}_{L^2\trkla{\Omega}}^2\leq Ch\,.
\end{multline}
As the strong convergence of $\phi\hx\tpm$ implies the existence of a subsequence converging pointwise almost everywhere towards $\phi$, we may use the continuity of $F$, the growth condition in \ref{item:potential}, and the dominated convergence theorem to obtain result.
The convergence expressed in \eqref{eq:convEG} can be established analogously.
\end{proof}

Considering test functions $\psi\h,\hat{\psi}\h\in L^2\trkla{0,T;\UhO}$ and $\eta\h\in L^2\trkla{0,T;\UhG}$ in \eqref{eq:modeldisc}, multiplying \eqref{eq:modeldisc:auxr} and \eqref{eq:modeldisc:auxs} by $\zeta\in L^2\trkla{0,T}$, and summing over all time steps, we rewrite \eqref{eq:modeldisc} using the time-interpolant functions:
\begin{subequations}\label{eq:timecont}
\begin{multline}\label{eq:timecont:phi1}
\iOmegaT \Ih{\para{t}\phi\hx\tl\psi\h}\dx\dt +\iOmegaT\nabla\mu\hx\tp\cdot\nabla\psi\h\dx\dt \\
+\beta^{-1}\iGammaT\IhG{\para{t}\phi\hx\tl\psi\h}\dG\dt +\beta^{-1}\iGammaT\nablaG\theta\hx\tp\cdot\nablaG\psi\h\dG\dt=0\,,
\end{multline}
\begin{multline}\label{eq:timecont:phi2}
\tfrac{1}{1+\xi}\iOmegaT\IhG{\para{t}\phi\hx\tl\eta\h}\dG\dt+\tfrac{1}{1+\xi}\iGammaT\nablaG\theta\hx\tp\cdot\nablaG\eta\h\dG\dt\\
+\tfrac{\xi}{1+\xi}\beta\iGammaT\IhG{\trkla{\beta\theta\hx\tp-\mu\hx\tp}\eta\h}\dG\dt=0\,,
\end{multline}
\begin{multline}\label{eq:timecont:pot}
\iOmegaT\Ih{\mu\hx\tp\hat{\psi}\h}\dx\dt+\iGammaT\IhG{\theta\hx\tp\hat{\psi}\h}\dG\dt=\iOmegaT\nabla\phi\hx\tp\cdot\nabla\hat{\psi}\h\dx\dt\\
+\iGammaT\nablaG\phi\hx\tp\cdot\nablaG\hat{\psi}\h\dG\dt +\int_0^T\frac{r\hx\tp}{\sqrt{E\h^{\Omega}\trkla{\phi\hx\tm}}}\iOmega\Ih{F^\prime\trkla{\phi\hx\tm}\hat{\psi}\h}\dx\dt\\
+\int_0^T\frac{s\hx\tp}{\sqrt{E\h^{\Gamma}\trkla{\phi\hx\tm}}}\iGamma\IhG{G^\prime\trkla{\phi\hx\tm}\hat{\psi}\h}\dx\dt
\end{multline}
\begin{align}
\int_0^T \para{t}r\hx\tl\zeta\trkla{t}\dt&=\int_0^T\zeta\trkla{t}\frac{1}{2\sqrt{E\h^\Omega\trkla{\phi\hx\tm}}}\iOmega\Ih{F^\prime\trkla{\phi\hx\tm}\para{t}\phi\hx\tl}\dx\dt\,,\label{eq:timecont:r}\\
\int_0^T \para{t}s\hx\tl\zeta\trkla{t}\dt&=\int_0^T\zeta\trkla{t}\frac{1}{2\sqrt{E\h^\Gamma\trkla{\phi\hx\tm}}}\iGamma\IhG{G^\prime\trkla{\phi\hx\tm}\para{t}\phi\hx\tl}\dG\dt\,.\label{eq:timecont:s}
\end{align}
\end{subequations}
By passing to the limit $\trkla{h,\tau,\xi}\rightarrow\trkla{0,0,\xi_0}$, we obtain the convergence result stated in the following theorem.
\begin{theorem}\label{th:convergence}
Let $d\in\tgkla{2,3}$ and let the assumptions \ref{item:T}, \ref{item:S1}, \ref{item:S2}, \ref{item:potential}, \ref{item:htau1}, and \ref{item:I} hold true.
Then, the triple $\trkla{\phi,\mu,\theta}$ obtained in Lemma \ref{lem:convergence} by passing to the limit $\trkla{h,\tau,\xi}\rightarrow\trkla{0,0,\xi_0}$ solves \eqref{eq:model} in the following weak sense:
\begin{subequations}\label{eq:weaksolution}
\begin{multline}\label{eq:weaksolution:phiO}
\iOmegaT\trkla{\phi_0-\phi}\para{t}\psi\dx\dt+\iOmegaT\nabla\mu\cdot\nabla\psi\dx\dt\\
+\beta^{-1}\iGammaT\trkla{\phi_0-\phi}\para{t}\psi\dx\dt+\beta^{-1}\iGammaT\nablaG\theta\cdot\nablaG\psi\dG\dt=0\,,
\end{multline}
\begin{multline}\label{eq:weaksolution:phiG}
\tfrac{1}{1+\xi_0}\iGammaT\trkla{\phi_0-\phi}\para{t}\eta\dG\dt+\tfrac1{1+\xi_0}\iGammaT\nablaG\theta\cdot\nablaG\eta\dG\dt\\
+\tfrac{\xi_0}{1+\xi_0}\beta\iGammaT\trkla{\beta\theta-\mu}\eta\dG\dt=0\,,
\end{multline}
\begin{multline}\label{eq:weaksolution:potential}
\iOmegaT\mu \hat{\psi}\dx\dt+\iGammaT\theta\hat{\psi}\dG\dt= \iOmegaT\nabla\phi\cdot\nabla\hat{\psi}\dx\dt+\iGammaT\nablaG\phi\cdot\nablaG\hat{\psi}\dG\dt\\
+ \iOmegaT {F}^\prime\trkla{\phi}\hat{\psi}\dx\dt+\iGammaT {G}^\prime\trkla{\phi}\hat{\psi}\dG\dt\,,
\end{multline}
\end{subequations}
for all $\psi\in H^1\trkla{0,T;\mathcal{V}}$ satisfying $\psi\trkla{\cdot,T}\equiv0$, $\eta\in H^1\trkla{0,T;H^1\trkla{\Gamma}}$ satsifying $\eta\trkla{\cdot,T}\equiv0$, and $\hat{\psi}\in L^2\trkla{0,T;\mathcal{V}}$.
\end{theorem}
\begin{proof}
For the convergence of \eqref{eq:timecont:phi1} and \eqref{eq:timecont:phi2} towards \eqref{eq:weaksolution:phiO} and \eqref{eq:weaksolution:phiG}, we refer the reader to the proof of Theorem 5.8 in \cite{KnopfLamLiuMetzger2021}, where the convergence was established by using $\Ih{\psi}$ with $\psi\in C^1\trkla{\tekla{0,T};C^\infty\trkla{\overline{\Omega}}}$ satisfying $\psi\trkla{\cdot,T}\equiv0$ and $\IhG{\eta}$ with $\eta\in C^1\trkla{\tekla{0,T};C^\infty\trkla{\Gamma}}$ satisfying $\eta\trkla{\cdot,T}\equiv0$ as test functions and applying assumption \ref{item:htau1}.\\

It remains to show \eqref{eq:weaksolution:potential}.
We shall start by passing to the limit in the evolution equations for $r\hx\tl$ and $s\hx\tl$, i.e.~in \eqref{eq:timecont:r} and \eqref{eq:timecont:s}.
This is intricate due to the lack of results concerning the convergence properties of $\para{t}\phi\hx\tl$.
Hence, we rely on the reconstruction of the chain-rule used in \eqref{eq:evolution:aux} for the formal derivation of the evolution equations, which allows us to integrate by parts and shift the time derivative onto the test function $\zeta$.
%Hence, our strategy for passing to the limit in the right-hand sides of \eqref{eq:timecont:r} and \eqref{eq:timecont:s} will rely on the reconstruction of the chain-rule used in \eqref{eq:evolution:aux} for the formal derivation of the evolution equation.
%This will then allow us to integrate by parts and shift the time derivative onto the test function $\zeta$.
Choosing $\zeta\in C^1\trkla{\tekla{0,T}}$ with $\zeta\trkla{T}=0$, we split the right-hand side of \eqref{eq:timecont:r} into
\begin{multline}\label{eq:split:r}
\int_0^T\zeta\trkla{t} \rkla{\frac1{2\sqrt{E\h^\Omega\trkla{\phi\hx\tm}}}-\frac1{2\sqrt{E\h^\Omega\trkla{\phi\hx\tl}}}}\iOmega\Ih{F^\prime\trkla{\phi\hx\tm}\para{t}\phi\hx\tl}\dx\dt\\
+\int_0^T\zeta\trkla{t}\frac{1}{2\sqrt{E\h^\Omega\trkla{\phi\hx\tl}}}\iOmega\Ih{\trkla{F^\prime\trkla{\phi\hx\tm}-F^\prime\trkla{\phi\hx\tl}}\para{t}\phi\hx\tl}\dx\dt\\
+\int_0^T\zeta\trkla{t}\frac{1}{2\sqrt{E\h^\Omega\trkla{\phi\hx\tl}}}\iOmega\Ih{F^\prime\trkla{\phi\hx\tl}\para{t}\phi\hx\tl}\dx\dt\\
=:D_1+D_2+D_3\,.
\end{multline}
Recalling \ref{item:potential}, \eqref{eq:tmp:EO}, Hölder's inequality, and the norm equivalence \eqref{eq:normequivalence}, we derive the estimate
\begin{align}
\begin{split}
\tabs{D_1}%\leq&\,\int_0^T\frac{\abs{E\h^\Omega\trkla{\phi\hx\tl}-E\h^\Omega\trkla{\phi\hx\tm}}}{2\sqrt{E\h^\Omega\trkla{\phi\hx\tl} E\h^\Omega\trkla{\phi\hx\tm}}\rkla{\sqrt{E\h^\Omega\trkla{\phi\hx\tl}}+\sqrt{E\h^\Omega\trkla{\phi\hx\tm}}}}\tabs{\zeta\trkla{t}}\iOmega\abs{\Ih{F^\prime\trkla{\phi\hx\tm}\para{t}\phi\h\tl}}\dx\dt \\
\leq &\,C\trkla{\gamma}\int_0^T\abs{E\hx^\Omega\trkla{\phi\hx\tl}-E\h^\Omega\trkla{\phi\hx\tm}}\norm{F^\prime\trkla{\phi\hx\tl}}_{L^2\trkla{\Omega}}\tau^{-1}\norm{\phi\hx\tp-\phi\hx\tm}_{L^2\trkla{\Omega}}\dt\,.
\end{split}
\end{align}
Recalling the growth condition on $F^\prime$ (cf.~\ref{item:potential}), we obtain
\begin{multline}
\abs{E\h^\Omega\trkla{\phi\hx\tl}-E\h^\Omega\trkla{\phi\hx\tm}}= \abs{\iOmega\Ih{F\trkla{\phi\hx\tl}-F\trkla{\phi\hx\tm}}\dx}\\
=\abs{ \iOmega\Ih{F^\prime\trkla{\varphi}\cdot\trkla{\phi\hx\tl-\phi\hx\tm}}\dx} \leq C\norm{\Ih{F^\prime\trkla{\varphi}}}_{L^2\trkla{\Omega}}\norm{\phi\hx\tl-\phi\hx\tm}_{L^2\trkla{\Omega}}\\
\leq C\rkla{\norm{\phi\hx\tl}_{L^6\trkla{\Omega}}^3+\norm{\phi\hx\tm}_{L^6\trkla{\Omega}}^3+1}\norm{\phi\hx\tp-\phi\hx\tl}_{L^2\trkla{\Omega}}\,,
\end{multline}
where $\varphi\trkla{\bs{x}_k}$ is in the convex hull of $\phi\hx\tl\trkla{\bs{x}_k}$ and $\phi\hx\tm\trkla{\bs{x}_k}$ on each vertex $\bs{x}_k$ ($k\in\tgkla{1,\ldots,\dim\UhO}$).
Combining the growth condition on $F^\prime$ with the regularity results collected in \eqref{eq:regularity:energy}, we conclude
%As $F^\prime$ is also a third order polynomial, we use \eqref{eq:regularity:energy} to conclude
\begin{align}
\begin{split}
\tabs{D_1}\leq&\,C\trkla{\gamma}\int_0^T\rkla{\norm{\phi\hx\tl}_{H^1\trkla{\Omega}}^3 +\norm{\phi\hx\tm}_{H^1\trkla{\Omega}}^3+1}^2\tau^{-1}\norm{\phi\hx\tp-\phi\hx\tm}_{L^2\trkla{\Omega}}^2\dt\\
\leq&\,C\trkla{\gamma}\tau^{1/2}\rkla{\tau^{-3/2}\norm{\phi\hx\tp-\phi\hx\tm}_{L^2\trkla{0,T;L^2\trkla{\Omega}}}^2}\leq C\trkla{\gamma}\tau^{1/2}\,.
\end{split}
\end{align}
To deal with the second integral in \eqref{eq:split:r}, we use \ref{item:potential}, Hölder's inequality, the norm equivalence \eqref{eq:normequivalence}, the regularity results stated in \eqref{eq:regularity:energy}, and Young's inequality to compute 
\begin{align}\label{eq:needC2}
\begin{split}
\tabs{D_2}\leq&\,C\trkla{\gamma}\int_0^T\norm{\Ih{F^\prime\trkla{\phi\h\tm}-F^\prime\trkla{\phi\h\tl}}}_{L^2\trkla{\Omega}}\tau^{-1}\norm{\phi\hx\tp-\phi\hx\tm}_{L^2\trkla{\Omega}}\dt\\
\leq&\,C\trkla{\gamma}\int_0^T\norm{\Ih{\rkla{\abs{\phi\hx\tl}^2+\abs{\phi\hx\tm}^2+1}\cdot\abs{\phi\hx\tl-\phi\hx\tm}}}_{L^2\trkla{\Omega}}\tau^{-1}\norm{\phi\hx\tp-\phi\hx\tm}_{L^2\trkla{\Omega}}\dt\\
\leq&\,C\trkla{\gamma}\int_0^T\rkla{\norm{\phi\hx\tl}_{L^6\trkla{\Omega}}^2+\norm{\phi\hx\tm}_{L^6\trkla{\Omega}}^2+1}\norm{\phi\hx\tl-\phi\hx\tm}_{L^6\trkla{\Omega}}\tau^{-1}\norm{\phi\hx\tp-\phi\hx\tm}_{L^2\trkla{\Omega}}\dt\\
\leq& C\trkla{\gamma}\int_0^T \tau^{1/4}\rkla{\tau^{-1}\norm{\phi\hx\tl-\phi\hx\tm}_{H^1\trkla{\Omega}}^2+\tau^{-3/2}\norm{\phi\hx\tp-\phi\hx\tm}_{L^2\trkla{\Omega}}^2}\dt\leq C\trkla{\gamma} \tau^{1/4}\,.
\end{split}
\end{align}
Therefore, $\tabs{D_2}$ vanishes for $\tau\searrow0$.
For the remaining term $D_3$, we can use the chain rule to obtain
\begin{align}
\begin{split}
D_3=&\,\int_0^T\zeta\trkla{t}\frac{1}{2\sqrt{E\h^\Omega\trkla{\phi\hx\tl}}}\frac{\mathrm{d}}{\dt}\iOmega\Ih{F\trkla{\phi\hx\tl}}\dx\dt=\,\int_0^T\zeta\trkla{t}\frac{\mathrm{d}}{\mathrm{d}t} \sqrt{E\h^\Omega\trkla{\phi\hx\tl}}\dt\\
=&-\int_0^T\para{t}\zeta\trkla{t} \sqrt{E\h^\Omega\trkla{\phi\hx\tl}}\dt - \zeta\trkla{0}\sqrt{E\h^\Omega\trkla{\phi\hx\tl}\big\vert_{t=0}}\\
&\rightarrow -\int_0^T\para{t}\zeta\trkla{t} \sqrt{\iOmega F\trkla{\phi}\dx}\dt-\zeta\trkla{0}\sqrt{\iOmega F\trkla{\phi_0}\dx}\,,
\end{split}
\end{align}
due to Corollary \ref{cor:convergenceE} and \ref{item:I}.
Integrating by parts with respect to time and recalling the definition of the initial conditions for $r\hx\tl$ defined in \ref{item:I}, we obtain from \eqref{eq:conv:r:weakstar}
\begin{align}\label{eq:weaksolution:r}
\int_0^T\para{t}\zeta\trkla{t} r\trkla{t}\dt=\int_0^T\para{t}\zeta\trkla{t}\sqrt{\iOmega F\trkla{\phi}\dx}\dt&&\text{with}~r\trkla{0}=\sqrt{\iOmega F\trkla{\phi_0}\dx}
\end{align}
for all $\zeta\in C^1\trkla{\tekla{0,T}}$ satisfying $\zeta\trkla{T}=0$.
Employing the density of $C^1\trkla{\tekla{0,T}}$ in $H^1\trkla{0,T}$, allows us to generalize this result to $\zeta\in H^1\trkla{0,T}$ satisfying $\zeta\trkla{T}=0$.\\
Analogous arguments allow us to pass to the limit in \eqref{eq:timecont:s} which provides
\begin{align}\label{eq:weaksolution:s}
\int_0^T\para{t}\zeta\trkla{t} s\trkla{t}\dt=\int_0^T\para{t}\zeta\trkla{t}\sqrt{\iGamma G\trkla{\phi}\dG}\dt&&\text{with}~s\trkla{0}=\sqrt{\iGamma G\trkla{\phi_0}\dG}
\end{align}
for all $\zeta\in H^1\trkla{0,T}$ satisfying $\zeta\trkla{T}=0$.\\

To pass to the limit in \eqref{eq:timecont:pot}, we choose $\hat{\psi}\h:=\Ih{\hat{\psi}}$ with $\hat{\psi}\in L^2\trkla{0,T;C^\infty\trkla{\overline{\Omega}}}$.
This, in particular, allows us to interchange the interpolation and the trace operator, i.e.~$\Ih{\hat{\psi}}\big\vert_{\Gamma_T}=\IhG{\hat{\psi}\big\vert_{\Gamma_T}}$.
Then, except for the last two terms the right-hand side, the convergence is a direct consequence of the convergence properties stated in Lemma \ref{lem:convergence}.
Looking at the penultimate term on the right-hand side, we use that
\begin{align}\label{eq:tmp:EO}
\abs{\frac{1}{\sqrt{E\h^\Omega\trkla{\phi\hx\tpm}}}-\frac{1}{\sqrt{\iOmega F\trkla{\phi}\dx}}}\leq  C\trkla{\gamma}\abs{E\h^\Omega\trkla{\phi\hx\tpm}-\iOmega F\trkla{\phi}\dx}\,,
\end{align}
i.e.~we obtain strong convergence in $L^p\trkla{0,T}$ for $p<\infty$ from Corollary \ref{cor:convergenceE}.
In order to pass to the limit in the remaining spatial integral, we mimic the proof of Corollary \ref{cor:convergenceE} and start by showing that the interpolation operator can be neglected.
On each $K\in\Th$ we obtain for the Lipschitz continuous part $F_1^\prime$
\begin{align}
\begin{split}
\int_K \abs{\trkla{1-\Ihop}\gkla{F_1^\prime\trkla{\phi\hx\tm}}}\dx \leq&\, \tabs{K}\norm{F_1^\prime\trkla{\phi\hx\tm}-\Ih{F_1^\prime\trkla{\phi\hx\tm}}}_{L^\infty\trkla{K}}\\
=&\,\tabs{K}\norm{F_1^\prime\trkla{\phi\hx\tm}-\sum_{i=0}^d \lambda_i F_1^\prime\trkla{\phi\hx\tm\trkla{\bs{x}_i}}}_{L^\infty\trkla{K}}\\
\leq&\, C\tabs{K} h\tabs{\nabla\phi\hx\tm}\leq C h\rkla{\norm{\nabla\phi\hx\tm}_{L^2\trkla{K}}^2+\tabs{K}}\,,
\end{split}
\end{align}
where $\tgkla{\bs{x}_i}_{i=0,\ldots,d}$ are the vertices of $K$ and $\tgkla{\lambda_i}_{i=0,\ldots,d}$ are the corresponding barycentric coordinates.
As the continuously differentiable part $F_2^\prime$ can be treated by arguments similar to the ones used in the proof of Corollary \ref{cor:convergenceE}, we may sum over all $K\in\Th$ and obtain for $\alpha\in[1,2)$
\begin{align}\label{eq:convergence:Fprime}
\begin{split}
&\norm{\iOmega\abs{\Ih{F^\prime\trkla{\phi\hx\tm}\hat{\psi}\h}- F^\prime\trkla{\phi}\hat{\psi}}\dx}_{L^\alpha\trkla{0,T}}\\
&\quad\leq \norm{\iOmega \abs{F^\prime\trkla{\phi\hx\tm}\hat{\psi}\h- F^\prime\trkla{\phi}\hat{\psi}}\dx}_{L^\alpha\trkla{0,T}} +\norm{\iOmega\abs{\trkla{1-\Ihop}\gkla{F^\prime\trkla{\phi\hx\tm}}\hat{\psi}\h}\dx}_{L^\alpha\trkla{0,T}} \\
&\qquad + \norm{\iOmega\abs{\trkla{1-\Ihop}\gkla{\Ih{F^\prime\trkla{\phi\hx\tm}}\hat{\psi}\h}}\dx}_{L^\alpha\trkla{0,T}}\\
%&\quad\leq \norm{\iOmega\abs{ F^\prime\trkla{\phi\hx\tm}\hat{\psi}\h- F^\prime\trkla{\phi}\hat{\psi}}\dx}_{L^\alpha\trkla{0,T}} +Ch \norm{\hat{\psi}\h}_{L^\alpha\trkla{0,T;L^\infty\trkla{\Omega}}} +C h\norm{\nabla\hat{\psi}\h}_{L^\alpha\trkla{0,T;L^2\trkla{\Omega}}} \\
&\quad\leq \norm{\iOmega \abs{F^\prime\trkla{\phi\hx\tm}\hat{\psi}\h- F^\prime\trkla{\phi}\hat{\psi}}\dx}_{L^\alpha\trkla{0,T}} +Ch
\end{split}
\end{align}
by applying Lemma \ref{lem:Ih}, an inverse estimate, \eqref{eq:regularity:energy:a}, and the stability of the interpolation operator.
As we have $\hat{\psi}\h\rightarrow\hat{\psi}$ in $L^2\trkla{0,T;L^\infty\trkla{\Omega}}$ and $F^\prime\trkla{\phi\hx\tm}\rightarrow F^\prime\trkla{\phi}$ in $L^p\trkla{0,T;L^1\trkla{\Omega}}$ ($p<\infty$), the right-hand side of \eqref{eq:convergence:Fprime} vanishes.
Therefore, recalling \eqref{eq:conv:r:weakstar}, we are able to pass to the limit in the penultimate term in \eqref{eq:timecont:pot}.
The treatment of the last term on the right-hand side of \eqref{eq:timecont:pot} follows by similar arguments, which provides
\begin{align}
\begin{split}
\iOmegaT&\mu\hat{\psi}\dx\dG+\iGammaT\theta\hat{\psi}\dG\dt= \iOmegaT\nabla\phi\cdot\nabla\hat{\psi}\dx\dt+\iGammaT\nablaG\phi\cdot\nablaG\hat{\psi}\dG\dt\\
&+\int_0^T\frac{r}{\sqrt{\iOmega F\trkla{\phi}\dx}}\iOmega F^\prime\trkla{\phi}\hat{\psi}\dx\dt+\int_0^T\frac{s}{\sqrt{\iGamma G\trkla{\phi}\dG}}\iGamma G^\prime\trkla{\phi}\hat{\psi}\dG\dt
\end{split}
\end{align}
for $\hat{\psi}\in L^2\trkla{0,T;C^\infty\trkla{\overline{\Omega}}}$.
Noting that the scalar auxiliary variables $r$ and $s$ are multiplied by sufficiently regular functions $\rkla{\iOmega F\trkla{\phi}\dx}^{-1/2}\iOmega F^\prime\trkla{\phi}\hat{\psi}\dx\in L^2\trkla{0,T}$ and $\rkla{\iGamma G\trkla{\phi}\dG}^{-1/2}\iGamma G^\prime\trkla{\phi}\hat{\psi}\dG\in L^2\trkla{0,T}$, which we may interpret as time derivatives of suitable $H^1\trkla{0,T}$ functions, we use \eqref{eq:weaksolution:r} and \eqref{eq:weaksolution:s} to substitute $r$ and $s$.
Finally, a density argument provides \eqref{eq:weaksolution:potential}.
\end{proof}

\begin{remark}\label{rem:uniqueness}
The weak solutions obtained in Theorem \ref{th:convergence} have the same regularity as the ones obtained in Theorem 5.8 in \cite{KnopfLamLiuMetzger2021}. 
In Therefore, it is also possible to show that $\para{t}\phi$ belongs to a suitable dual space, which reduces the requirements on the test function further, and that the solution obtained in Theorem \ref{th:convergence} by passing to the limit $\trkla{h,\tau,\xi}\rightarrow\trkla{0,0,\xi_0}$ is unique (cf.~Corollary 5.9 in \cite{KnopfLamLiuMetzger2021}).
%Comparing the assumptions needed to obtain the convergence results, we note that out assumptions on the relation of $h$ and $\tau$ are stricter than the one used in \cite{KnopfLamLiuMetzger2021}, as we need the additional assumption \ref{item:htau2} to pass to the limit in the additional SAV equations \eqref{eq:timecont:r} and \eqref{eq:timecont:s} (cf.~\eqref{eq:needC2}).
\end{remark}

\section{Numerical simulations}\label{sec:numerics}
In this section, we demonstrate the practicality of the proposed SAV scheme \eqref{eq:modeldisc} by revisiting test cases discussed in \cite{Metzger2021a} and \cite{KnopfLamLiuMetzger2021} for algorithms based on convex-concave splittings and comparing the results.
For this reason, we adapted said algorithm to the proposed linear scheme and implemented it in the \texttt{C++} framework EconDrop (cf.~\cite{Campillo2012, Grun2013c, GrunGuillenMetzger2016, Metzger_2018b, Metzger2018, Metzger2021a}).
\subsection{Phase separation}\label{subsec:separation}
We start by simulating a phase separation process in the case of vanishing adsorption rates, i.e.~$\xi=0$ (cf.~Scenario 2 in \cite{Metzger2021a}).
In this scenario, we consider the domain $\Omega:=\trkla{0,1}^2$ and the initial configuration
\begin{align}
\phi_0\,:\,\mathds{R}^2\supset\Omega\rightarrow\tekla{-1,+1}\quad\trkla{x_1,x_2}\mapsto {\max\gkla{0.1\sin\trkla{\pi x_1},0.1\sin\trkla{\pi x_2}}}\,.
\end{align}
As this initial configuration is unstable, the two phases will separate into different areas where $\phi$ attains values close to $\pm1$.
This evolution is depicted in Fig.~\ref{fig:separation}.
Thereby, the light colour represents the phase $\phi=+1$ and the dark colour represents the phase $\phi=-1$. 
In order to visualize the initial condition, Fig.~\ref{fig:separation:a} is rescaled such that the dark colour represents $\phi=0$ and the light colour corresponds to $\phi=0.1$.
For the double-well potentials $F$ and $G$, we use the shifted polynomial double-well potential $\widetilde{W}_{\operatorname{pol}}\trkla{\phi}:=\tfrac14\trkla{1-\phi^2}^2+c_0$ with an additive constant $c_0$ which is chosen as $0.01/\tabs{\Omega}$ for $F$ and $0.01/\tabs{\Gamma}$ for $G$.
The remaining parameters can be found in Table \ref{tab:param1}.\\
\ifnum\value{IMASTYLE}>1 {
\begin{table}[]
\tblcaption{Parameters used in Section \ref{subsec:separation}.}
{%
\begin{tabular}{@{}cccccccc@{}}
\tblhead{$\mobO$&$\varepsilon$& $\sigma$& $\mobG$&$\delta$& $\beta$&$\xi$&$T$}
0.01&0.02&2&0.02&0.02&1&0&1.0
\lastline
\end{tabular}
}
\label{tab:param1}
\end{table}
} \else {
\begin{table}
\center
\begin{tabular}{c|c|c|c|c|c|c|c}
$\mobO$&$\varepsilon$& $\sigma$& $\mobG$&$\delta$& $\beta$&$\xi$&$T$\\
0.01&0.02&2&0.02&0.02&1&0&1.0
\end{tabular}
\caption{Parameters used in Section \ref{subsec:separation}.}
\label{tab:param1}
\end{table}
}\fi
\begin{figure}
\begin{center}
\newcommand{\scale}{.23\textwidth}
\subfloat[][$t=0$ (rescaled)]{
\includegraphics[width=\scale]{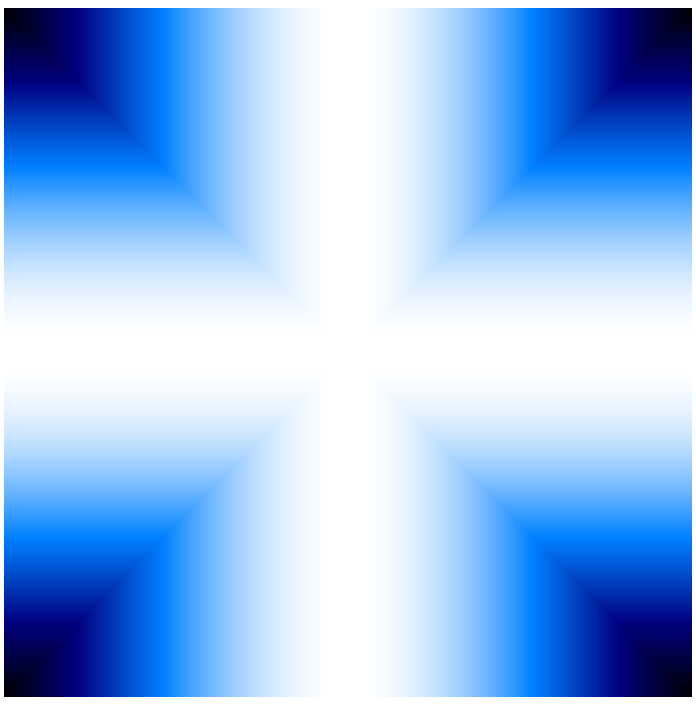}
\label{fig:separation:a}
}\hfill
\subfloat[][$t=0.01$]{
\includegraphics[width=\scale]{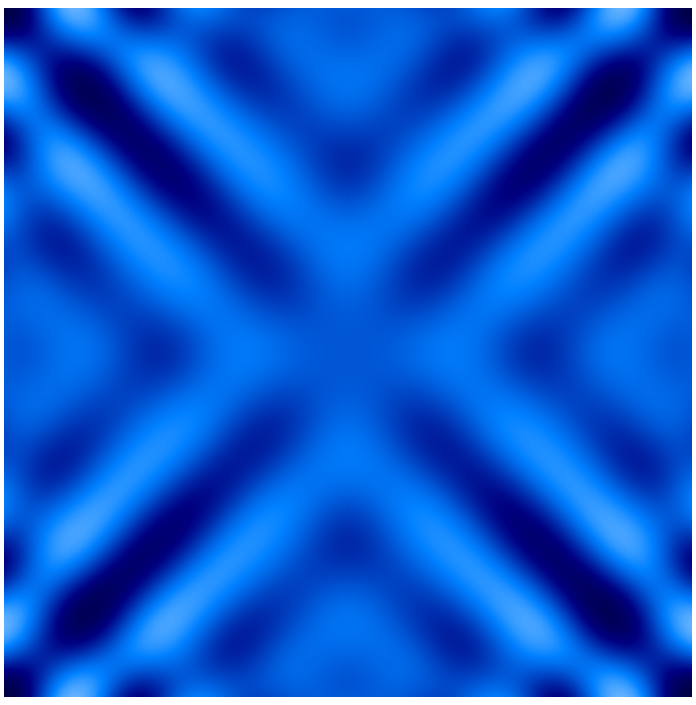}
}\hfill
\subfloat[][$t=0.03$]{
\includegraphics[width=\scale]{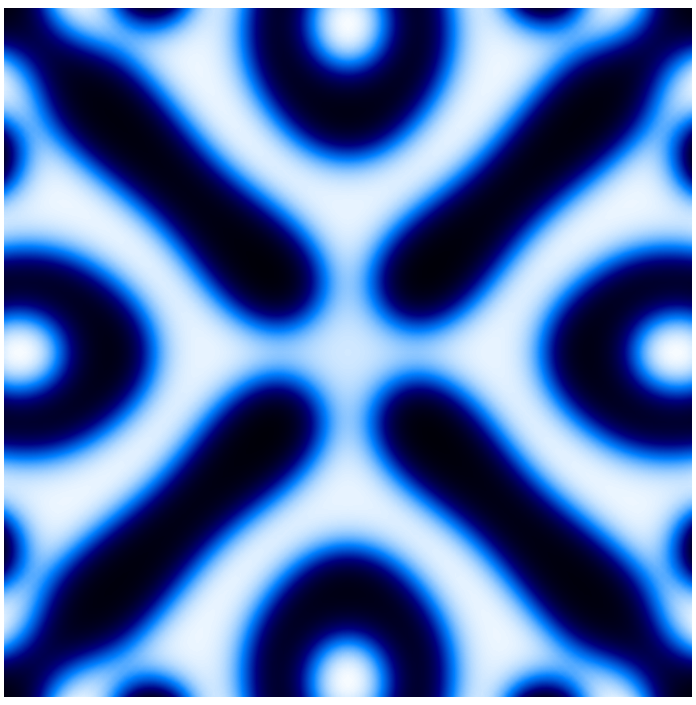}
}\hfill
\subfloat[][$t=0.06$]{
\includegraphics[width=\scale]{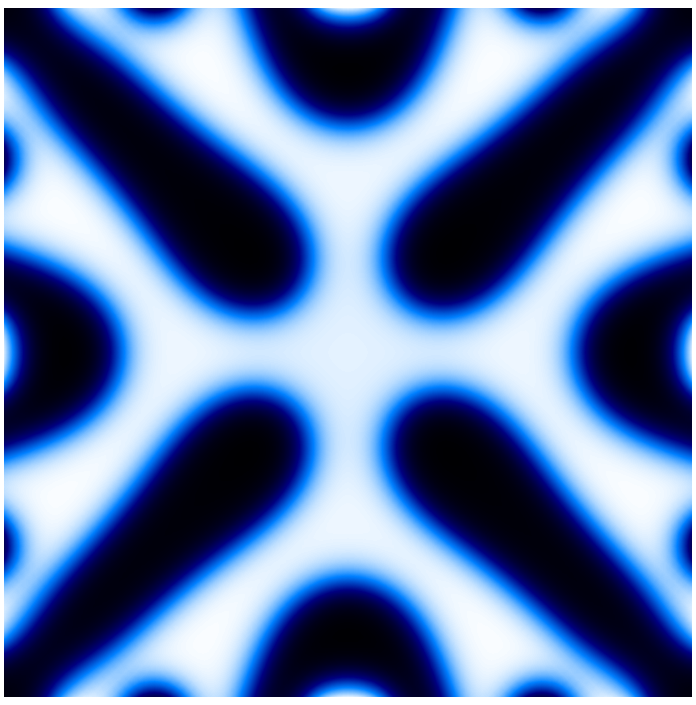}
}\\
\subfloat[][$t=0.12$]{
\includegraphics[width=\scale]{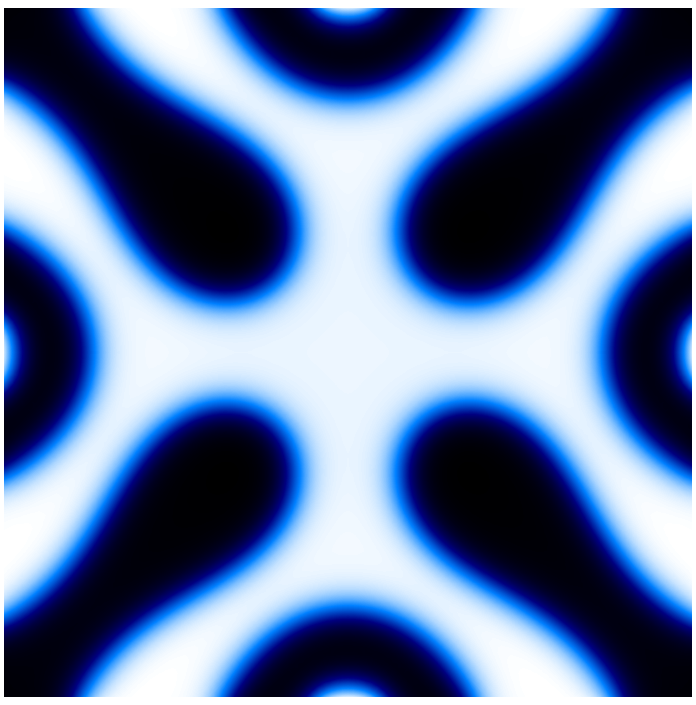}
}\hfill
\subfloat[][$t=0.3$]{
\includegraphics[width=\scale]{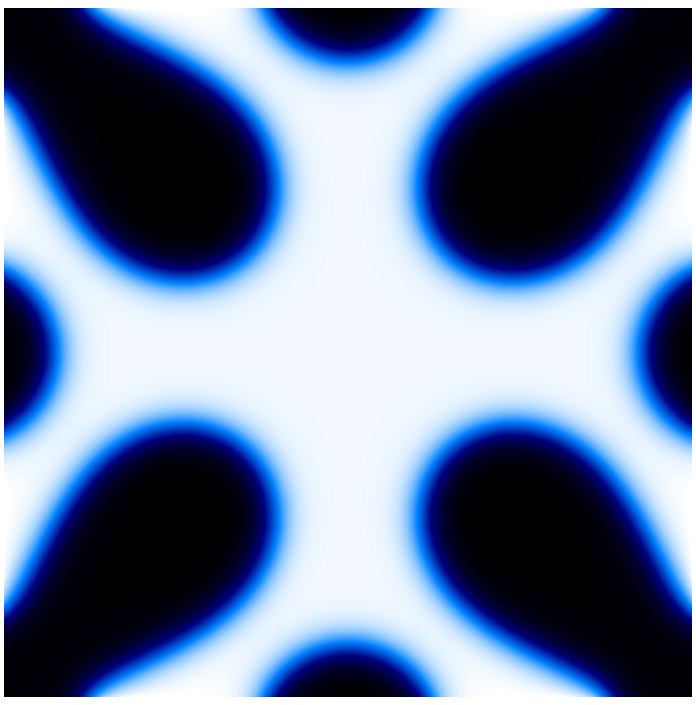}
}\hfill
\subfloat[][$t=0.6$]{
\includegraphics[width=\scale]{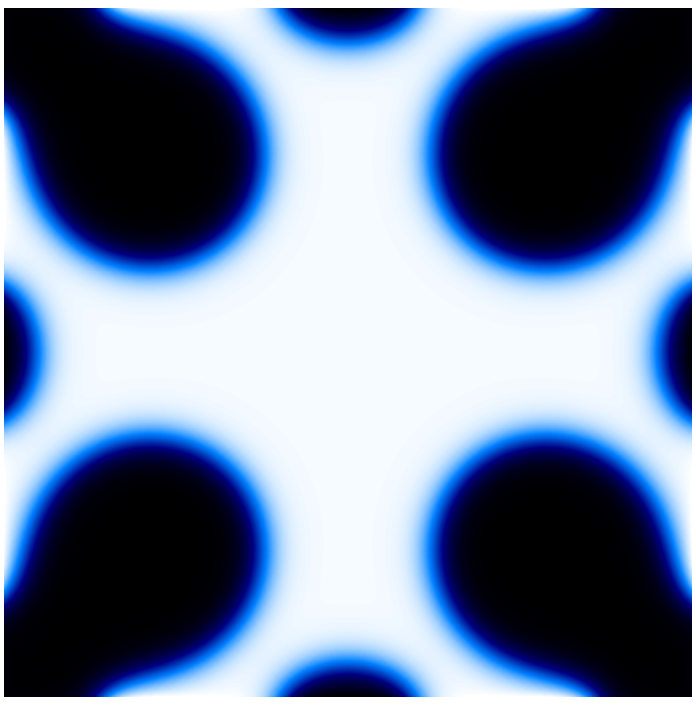}
}\hfill
\subfloat[][$t=1$]{
\includegraphics[width=\scale]{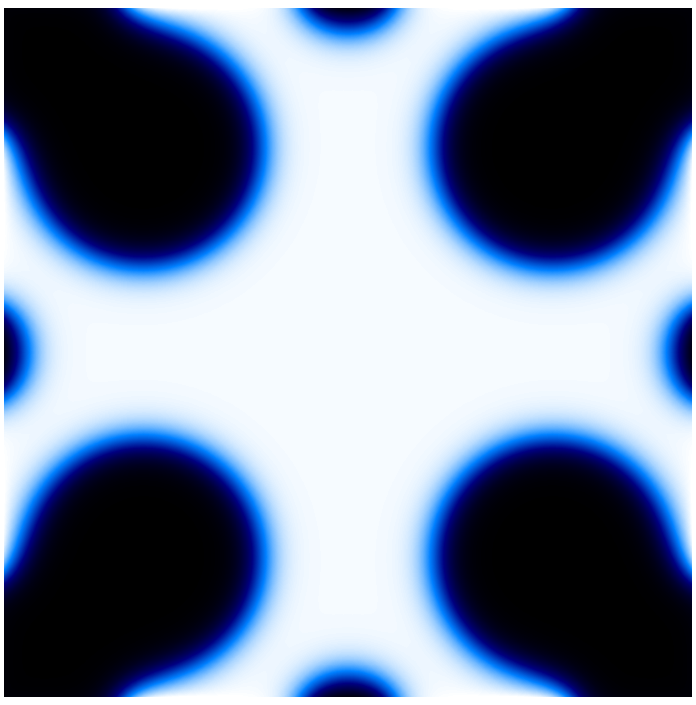}
}
\caption{Visualization of the phase-separation process investigated in Sec.~\ref{subsec:separation} using $h=\sqrt{2}\cdot 2^{-7}$ and $\tau=2\cdot10^{-5}$.}
\label{fig:separation}
\end{center}
\end{figure}
To compute experimental orders of convergence w.r.t.~$h$ and $\tau$, we conduct this simulation with $\tau=k\cdot10^{-5}$, $k\in\tgkla{1,2,4}$, and $h=\sqrt{2}\cdot 2^{-l}$, $l\in\tgkla{6,7,8}$.
The considered mesh sizes correspond to $\dim\UhO\in\tgkla{4225,\,16641,\,66049}$ and $\dim\UhG\in\tgkla{256,\,512,\,1024}$.
Fixing $\tau=2\cdot 10^{-5}$, we treat the solution computed on the finest mesh ($h=\sqrt{2}\cdot 2^{-8}$) as reference solution.
We then use the $L^2\trkla{0,T;L^2\trkla{\Omega}}$-norm (or the $L^2\trkla{0,T;L^2\trkla{\Gamma}}$-norm, respectively) of the deviations of the solutions computed on the coarser meshes from the reference solution to obtain experimental orders of convergence (EOC) w.r.t.~to $h$.
Thereby, the integration in time is approximated by a trapezoidal rule with step size $2\cdot10^{-4}$.
As shown in Tab.~\ref{tab:eoc:h}, we obtain order $2.28$ for the convergence of $\phi\hx\tl$ on $\Omega$ and $1.13$ for the convergence of $\phi\hx\tl\big\vert_\Gamma$ on $\Gamma$.
We want to point out that these results are identical to the ones obtained in \cite{Metzger2021a} for a convex-concave splitting scheme.\\
In a similar manner, we obtain experimental orders of convergence w.r.t.~$\tau$ for fixed $h=\sqrt{2}\cdot 2^{-7}$.
As depicted in Tab.~\ref{tab:eoc:tau}, we obtain order $1.59$ in $\Omega$ and order $1.58$ on $\Gamma$.
Again, these results match with the ones obtained in \cite{Metzger2021a}.
\ifnum\value{IMASTYLE}>1 {
\begin{table}[]
\tblcaption{Experimental order of convergence of $\phi$ and $\phi\big\vert_\Gamma$ w.r.t.~$h$.}
{%
\begin{tabular}{@{}ccccccc@{}}
\tblhead{$h$&\qquad$\phi$\qquad&$\norm{\cdot}_{L^2\trkla{0,T;L^2\trkla{\Omega}}}$ & $\text{EOC}_h^\Omega$ &\qquad$\phi\big\vert_\Gamma$\qquad& $\norm{\cdot}_{L^2\trkla{0,T;L^2\trkla{\Gamma}}}$ & $\text{EOC}_h^\Gamma$}
$\sqrt{2}\cdot2^{-7}$& & 6.28E-03&-& & 8.33E-02&-\\
$\sqrt{2}\cdot2^{-6}$& & 3.06E-02& 2.28& & 1.82E-01& 1.13
\lastline
\end{tabular}
}
\label{tab:eoc:h}
\end{table}
} \else {

\begin{table}
\begin{center}
\subfloat[][EOC w.r.t.~$h$ on $\Omega$.]{
\begin{tabular}{c|cc}
\diagbox{$h$}{$\phi\phantom{\big\vert_\Gamma}$} & $\norm{\cdot}_{L^2\trkla{0,T;L^2\trkla{\Omega}}}$ & $\text{EOC}_h^\Omega$\\\hline
$\sqrt{2}\cdot2^{-7}$ & 6.28E-03&-\\
$\sqrt{2}\cdot2^{-6}$ & 3.06E-02& 2.28
\end{tabular}}\hfill

\subfloat[][EOC w.r.t.~$h$ on $\Gamma$.]{
\begin{tabular}{c|cc}
\diagbox{$h$}{$\phi\big\vert_\Gamma$} & $\norm{\cdot}_{L^2\trkla{0,T;L^2\trkla{\Gamma}}}$ & $\text{EOC}_h^\Gamma$\\\hline
$\sqrt{2}\cdot2^{-7}$ & 8.33E-02&-\\
$\sqrt{2}\cdot2^{-6}$ & 1.82E-01& 1.13
\end{tabular}}
\caption{Experimental order of convergence w.r.t.~$h$.}
\label{tab:eoc:h}
\end{center}
\end{table}
}\fi

\ifnum\value{IMASTYLE}>1 {
\begin{table}[]
\tblcaption{Experimental order of convergence of $\phi$ and $\phi\big\vert_\Gamma$ w.r.t.~$\tau$.}
{%
\begin{tabular}{@{}ccccccc@{}}
\tblhead{$\tau$&\qquad$\phi$\qquad&$\norm{\cdot}_{L^2\trkla{0,T;L^2\trkla{\Omega}}}$ & $\text{EOC}_\tau^\Omega$ &\qquad$\phi\big\vert_\Gamma$\qquad& $\norm{\cdot}_{L^2\trkla{0,T;L^2\trkla{\Gamma}}}$ & $\text{EOC}_\tau^\Gamma$}
$2\cdot10^{-5}$& & 4.79E-03&-& & 2.48E-02&-\\
$4\cdot10^{-5}$& & 1.44E-02& 1.59& & 7.38E-02 & 1.58
\lastline
\end{tabular}
}
\label{tab:eoc:tau}
\end{table}
} \else {
\begin{table}
\begin{center}
\subfloat[][EOC w.r.t.~$\tau$ on $\Omega$.]{
\begin{tabular}{c|cc}
\diagbox{$\tau$}{$\phi\phantom{\big\vert_\Gamma}$} & $\norm{\cdot}_{L^2\trkla{0,T;L^2\trkla{\Omega}}}$ & $\text{EOC}_\tau^\Omega$\\\hline
$2\cdot10^{-5}$ & 4.79E-03&-\\
$4\cdot10^{-5}$ & 1.44E-02& 1.59
\end{tabular}}\hfill
\subfloat[][EOC w.r.t.~$\tau$ on $\Gamma$.]{
\begin{tabular}{c|cc}
\diagbox{$\tau$}{$\phi\big\vert_\Gamma$} & $\norm{\cdot}_{L^2\trkla{0,T;L^2\trkla{\Gamma}}}$ & $\text{EOC}_\tau^\Gamma$\\\hline
$2\cdot10^{-5}$ & 2.48E-02&-\\
$4\cdot10^{-5}$ & 7.38E-02 & 1.58
\end{tabular}}
\caption{Experimental order of convergence w.r.t.~$\tau$.}
\label{tab:eoc:tau}
\end{center}
\end{table}
}\fi
In order to compare the results obtained from the linear SAV scheme \eqref{eq:modeldisc} with the ones obtained via the non-linear convex-concave splitting scheme used in \cite{Metzger2021a}, we plotted for $h=\sqrt{2}\cdot 2^{-7}$ and $\tau=2\cdot10^{-5}$ the modified energy
%To compare the results obtained with the SAV scheme \eqref{eq:modeldisc}, we compute for $h=\sqrt{2}\cdot 2^{-7}$ and $\tau=2\cdot10^{-5}$ the modified energy
\begin{subequations}
\begin{align}
\widetilde{\mathcal{E}\h}\trkla{\phi\hx\tp,r\hx\tp,s\hx\tp}:=\tfrac12\varepsilon\sigma\iOmega\tabs{\nabla\phi\hx\tp}^2\dx+\varepsilon^{-1}\sigma\tabs{r\hx\tp}^2+\tfrac12\delta\iGamma\tabs{\nablaG\phi\hx\tp}^2\dG+\delta^{-1}\tabs{s\hx\tp}^2
\end{align}
(cf.~Rem.~\ref{rem:energy}) and the original discrete energy
\begin{multline}
\mathcal{E}\h\trkla{\phi\hx\tp}:=\tfrac12\varepsilon\sigma\iOmega\tabs{\nabla\phi\hx\tp}^2\dx+\varepsilon^{-1}\sigma\iOmega\Ih{F\trkla{\phi\hx\tp}}\dx +\tfrac12\delta\iGamma\tabs{\nablaG\phi\hx\tp}^2\dG\\
+\delta^{-1}\iGamma\IhG{G\trkla{\phi\hx\tp}}\dG\,
\end{multline}
\end{subequations}
in Fig.~\ref{fig:energy_cross}.
Thereby, the modified energy obtained from the SAV scheme is denoted by $\widetilde{\mathcal{E}}_{h,\text{SAV}}$, the energy $\mathcal{E}\h$ reconstructed from the results of the SAV scheme by $\mathcal{E}_{h,\text{SAV}}$, and the energy computed via the convex-concave splitting scheme by $\mathcal{E}_{h,\text{cc}}$.
As the congruent lines in Fig.~\ref{fig:energy_cross} indicate, the results are identical for the considered set of parameters.
%As depicted in Fig.~\ref{fig:energy_cross}, for the considered set of parameters the energies are identical.
\begin{figure}
\begin{center}
\includegraphics[width=0.4\textwidth]{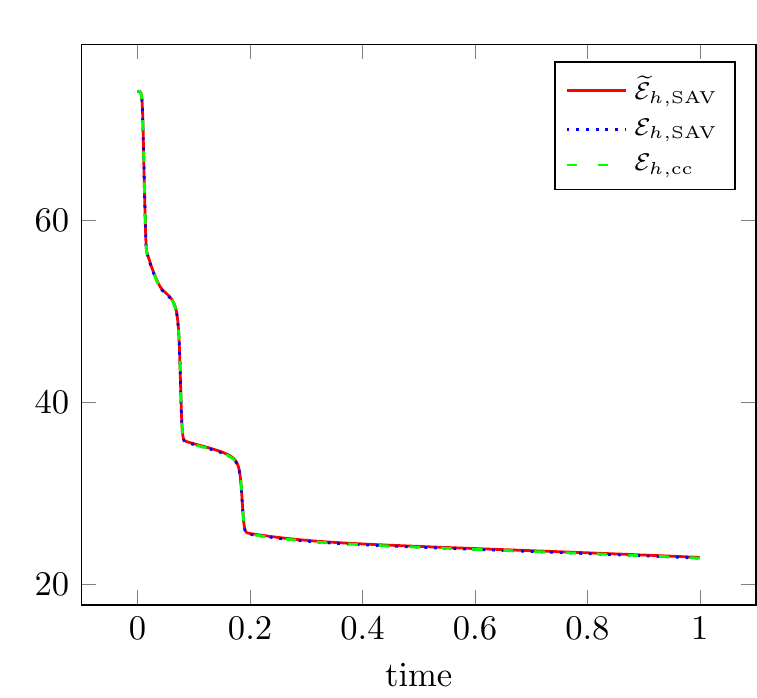}
\caption{Decrease of energy.}
\label{fig:energy_cross}
\end{center}
\end{figure}
\subsection{Adsorption processes}\label{subsec:adsorption}
Next, we test our discrete scheme on the setting proposed in \cite{KnopfLamLiuMetzger2021} which focuses on the influence of the adsorption rate $\xi$.\\
Again, we consider the domain $\Omega:=\trkla{0,1}^2\subset\mathds{R}^2$.
On its left boundary, we place an elliptical shaped droplet with barycenter at $\trkla{0.1,0.5}$, a maximal horizontal elongation of $0.6814$, and a maximal vertical elongation of $0.367$ (see Fig.~\ref{fig:adsorption:initial}).
As we will focus on the adsorption processes, we use very fine spatial and temporal discretizations.
In particular, the spatial discretization parameter $h$ is set to $\sqrt{2}\cdot 2^{-8}$ resulting in $\dim\UhO=66049$ and $\dim\UhG=1024$, while the time interval $\trkla{0,2.5}$ is divided into subintervals of length $\tau=3\cdot10^{-5}$.
The potentials $F$ and $G$ are chosen as shifted polynomial double-well potentials with additive constants $0.001/\tabs{\Omega}$ and $0.001/\tabs{\Gamma}$, respectively.
The other parameters are collected in Tab.~\ref{tab:param2}.\\
The evolution of the droplet is visualized in Fig.~\ref{fig:adsorption:evolution} for $\xi\in\tgkla{0,1,\infty}$.
As for $\xi=0$ the integral of $\phi$ is conserved on $\Omega$ and on $\Gamma$ individually (cf.~\eqref{eq:individualconservation}), the contact area can not change.
However, the elliptical droplet still tries to attain an energetically advantageous shape with constant mean curvature (see Fig.~\ref{fig:adsorption:evolution:xi0}).
For $\xi>0$, this constrained is relaxed to a combined conservation of mass (cf.~\eqref{eq:conservation}).
This allows the contact area to grow (see Fig.~\ref{fig:adsorption:evolution:xi1} and \ref{fig:adsorption:evolution:xiinfty}) and thereby to reduce the total energy even further (cf.~Fig.~\ref{fig:energy_xi}).\\
As in the last section, we compare the results obtained from the SAV scheme with the results obtained by the non-linear scheme discussed in \cite{KnopfLamLiuMetzger2021} by computing the energies. Again, we note that for the considered set of parameters, the energy plots align perfectly.\\
%A comparison of the energy computed for $\xi=1$ via the proposed SAV scheme with the energy $\mathcal{E}_{h,\text{cc}}$ computed with a convex-concave splitting scheme shows perfect alignment.\\
As in \cite{KnopfLamLiuMetzger2021}, we compute the experimental order of convergence for $\xi\searrow0$ and $\xi\nearrow\infty$.
The results are collected in Tab.~\ref{tab:EOC:xi0} and Tab.~\ref{tab:EOC:xiinfty}.
Although the absolute errors are larger than the ones computed in \cite{KnopfLamLiuMetzger2021}, the computed orders of convergence are the same.
\ifnum\value{IMASTYLE}>1 {
\begin{table}[]
\tblcaption{Parameters used in Section \ref{subsec:adsorption}.}
{%
\begin{tabular}{@{}cccccccc@{}}
\tblhead{$\mobO$&$\varepsilon$& $\sigma$& $\mobG$&$\delta$& $\beta$&$\xi$&$T$}
0.01&0.01&2&0.02&0.01&4&$\in\tekla{0,\infty}$&2.5
\lastline
\end{tabular}
}
\label{tab:param2}
\end{table}
} \else {
\begin{table}
\center
\begin{tabular}{c|c|c|c|c|c|c|c}
$\mobO$&$\varepsilon$& $\sigma$& $\mobG$&$\delta$& $\beta$&$\xi$&$T$\\
0.01&0.01&2&0.02&0.01&4&$\in\tekla{0,\infty}$&2.5
\end{tabular}
\caption{Parameters used in Section \ref{subsec:adsorption}.}
\label{tab:param2}
\end{table}
}\fi
\begin{figure}
\begin{center}
\includegraphics[width=0.3\textwidth]{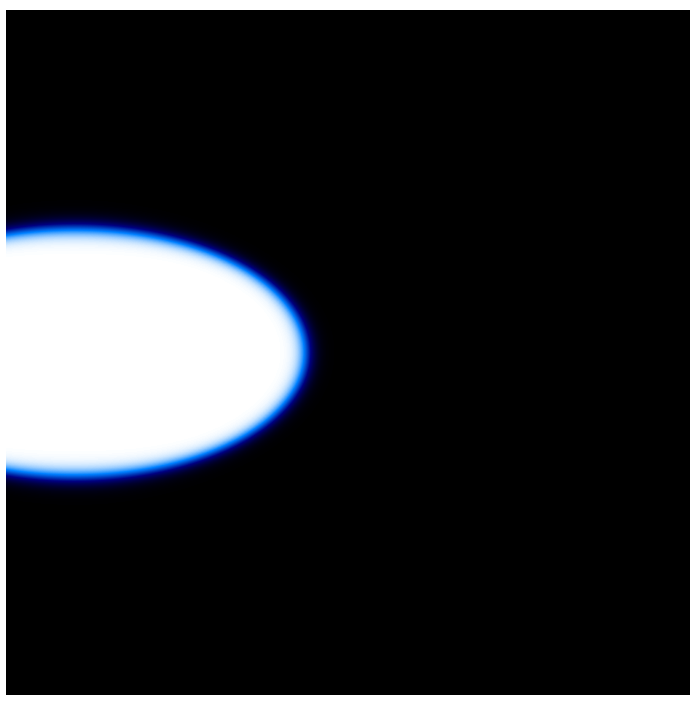}
\caption{Initial datum used in Sec.~\ref{subsec:adsorption}.}
\label{fig:adsorption:initial}
\end{center}
\end{figure}
\begin{figure}
\begin{center}
\newcommand{\scale}{.23\textwidth}
\subfloat[][$\xi=0$]{
\includegraphics[width=\scale]{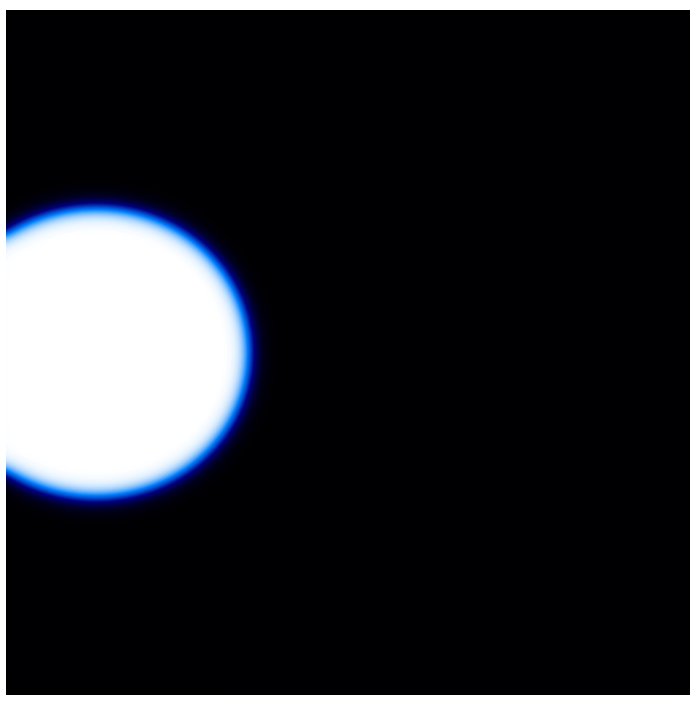}

\includegraphics[width=\scale]{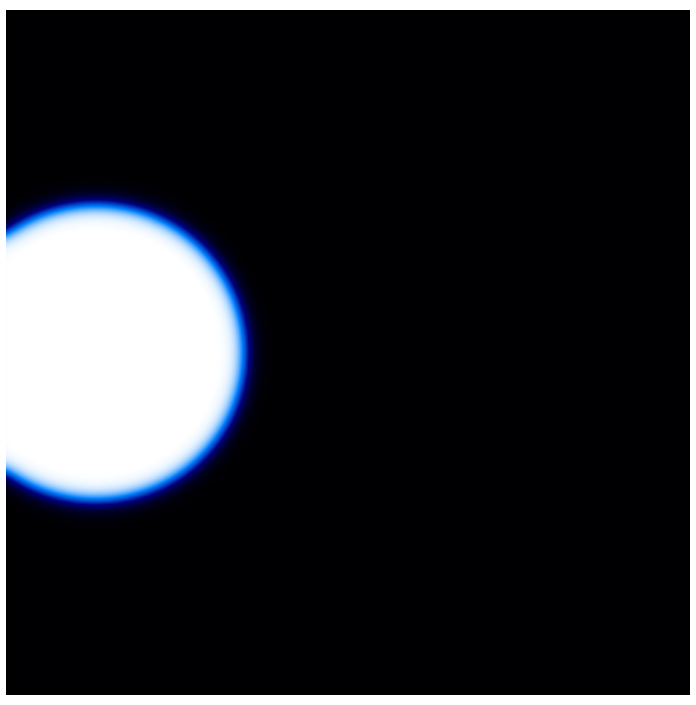}

\includegraphics[width=\scale]{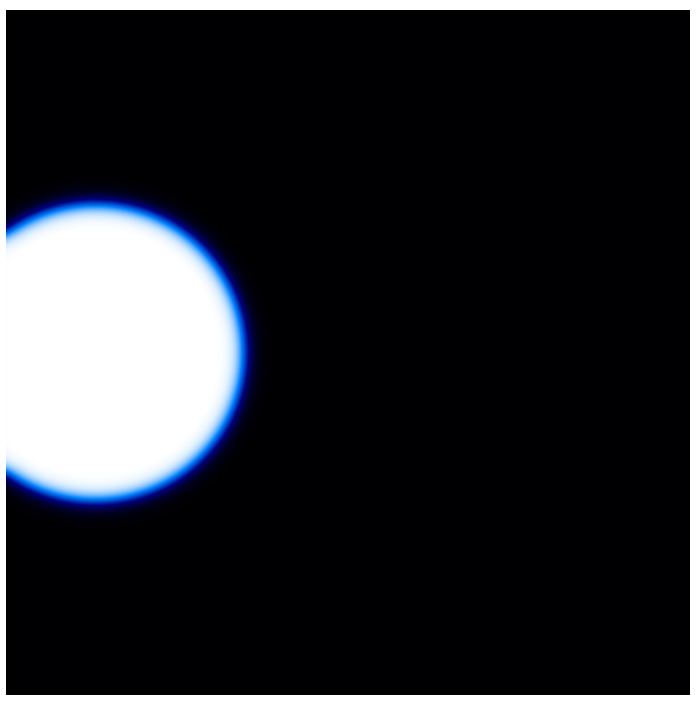}

\includegraphics[width=\scale]{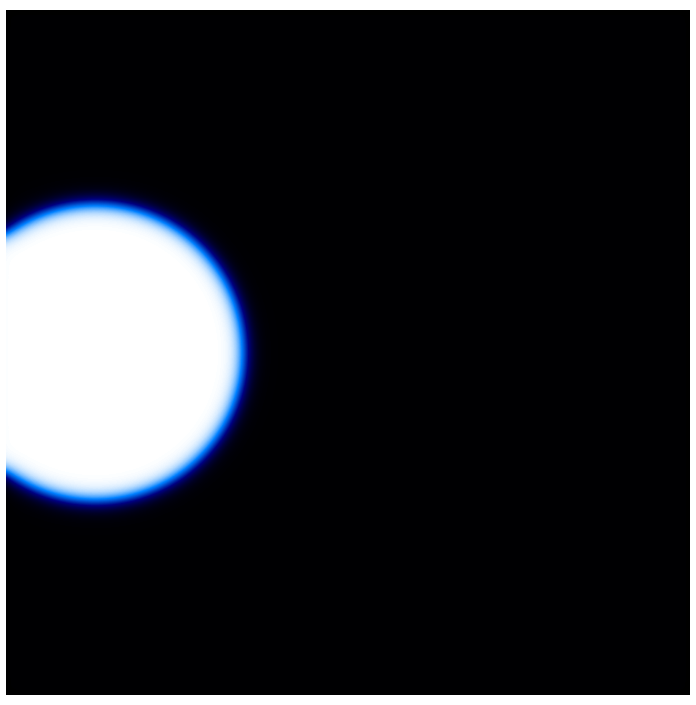}
\label{fig:adsorption:evolution:xi0}
}\\

\subfloat[][$\xi=1$]{
\includegraphics[width=\scale]{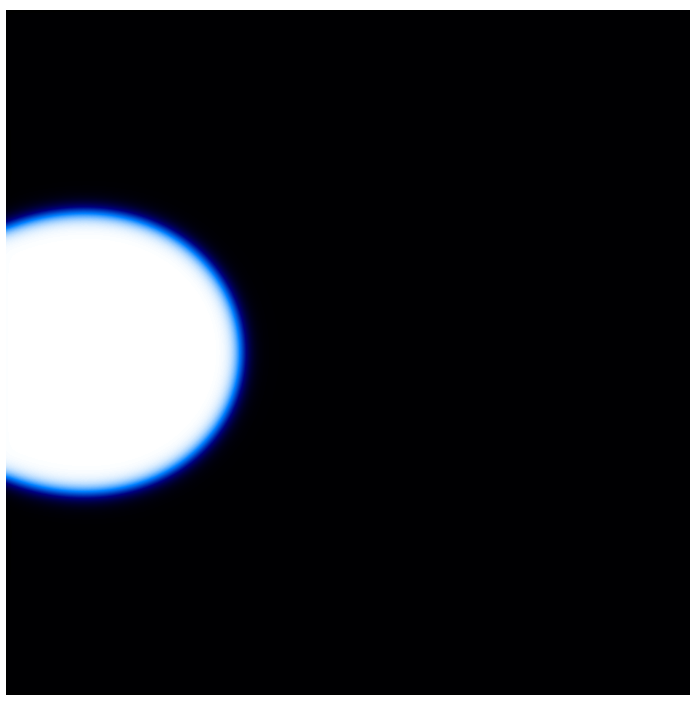}

\includegraphics[width=\scale]{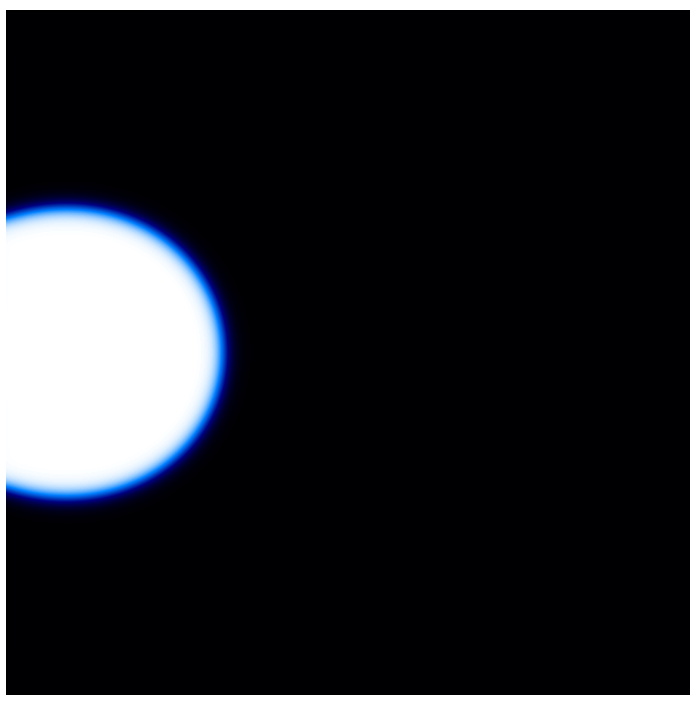}

\includegraphics[width=\scale]{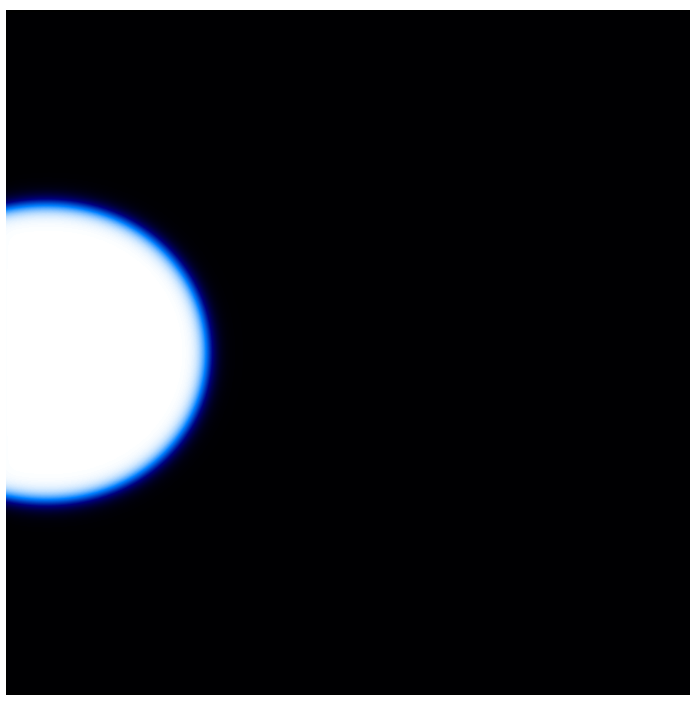}

\includegraphics[width=\scale]{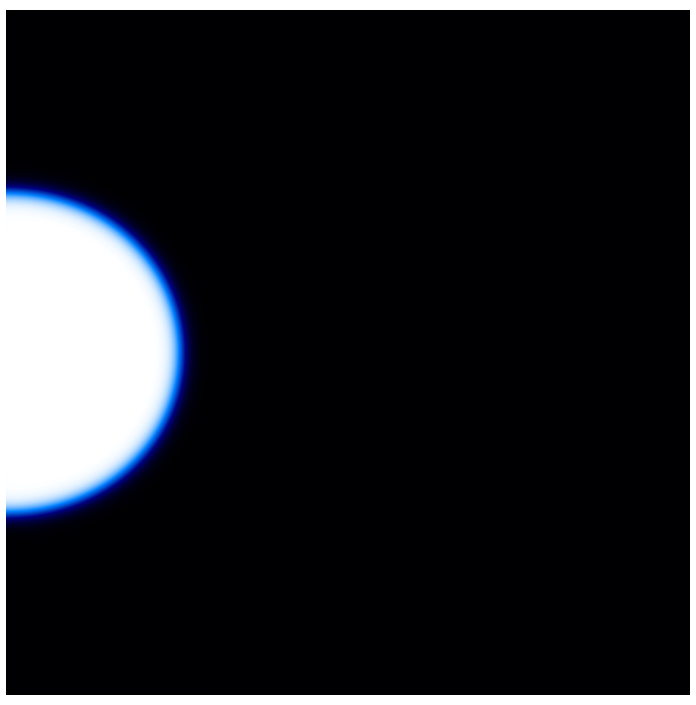}
\label{fig:adsorption:evolution:xi1}
}\\
\subfloat[][$\xi=\infty$]{
\includegraphics[width=\scale]{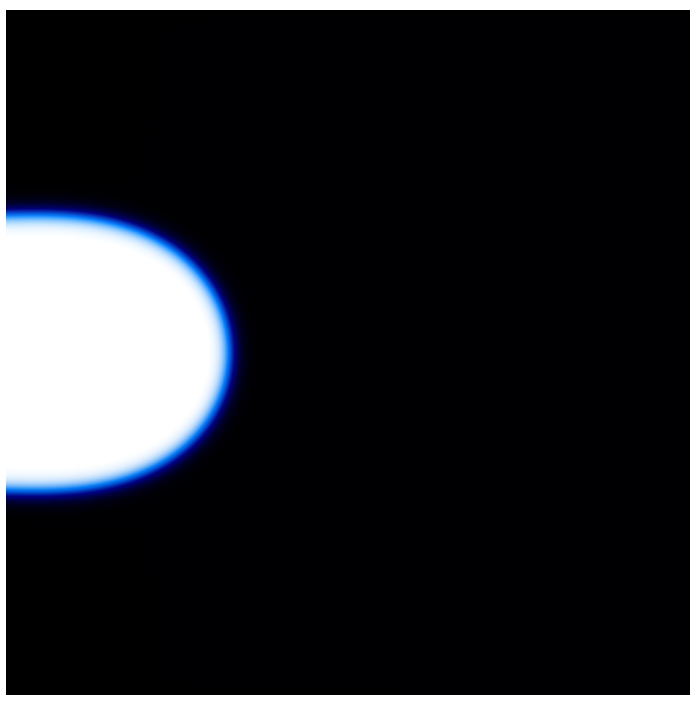}

\includegraphics[width=\scale]{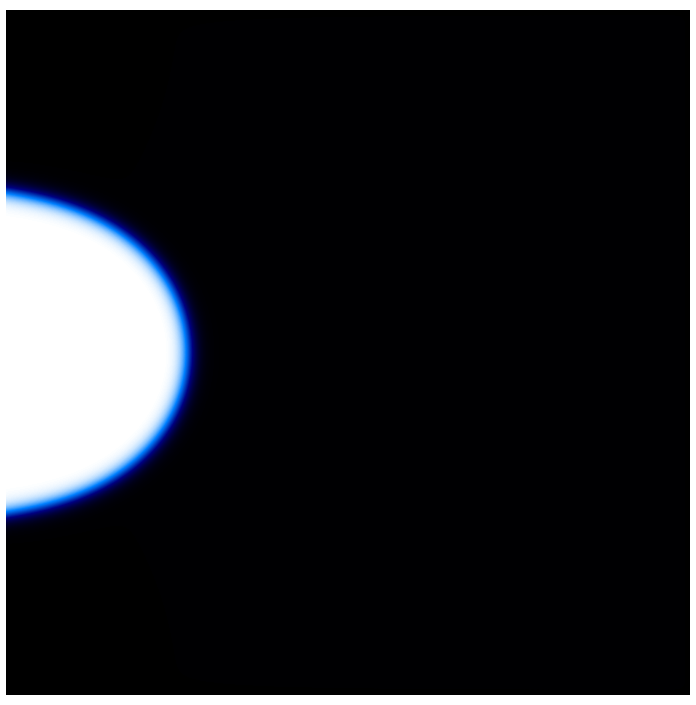}

\includegraphics[width=\scale]{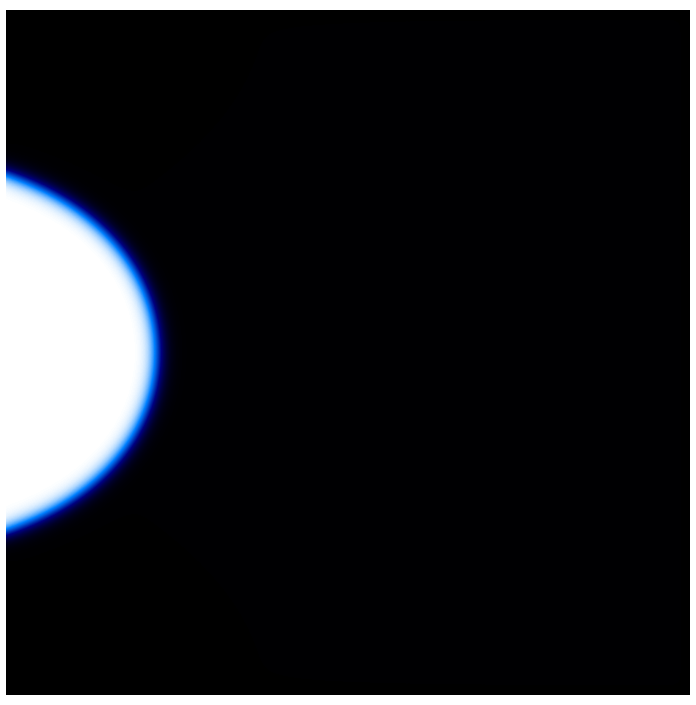}

\includegraphics[width=\scale]{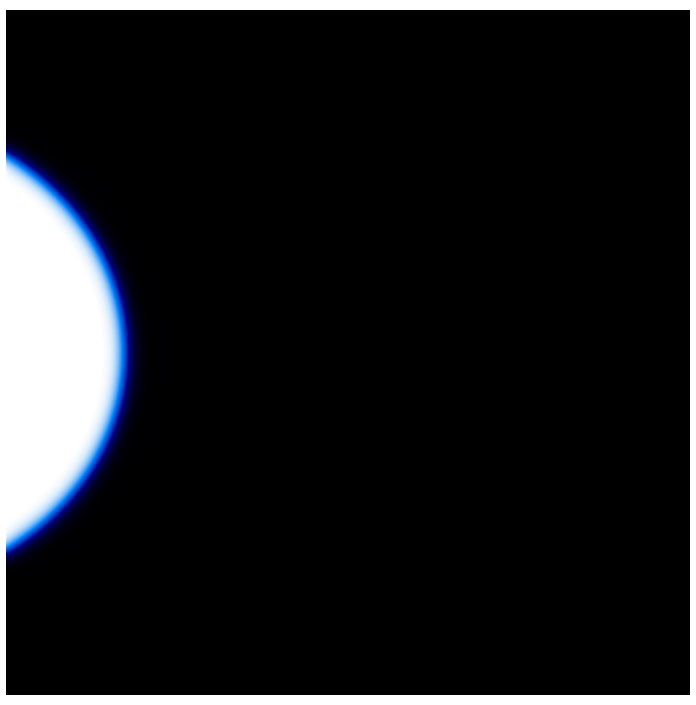}
\label{fig:adsorption:evolution:xiinfty}
}
\caption{Phase-field at $t=0.5$, $t=1.0$, $t=1.5$, and $t=2.5$.}
\label{fig:adsorption:evolution}
\end{center}
\end{figure}

\begin{figure}
\begin{center}
\includegraphics[width=0.3\textwidth]{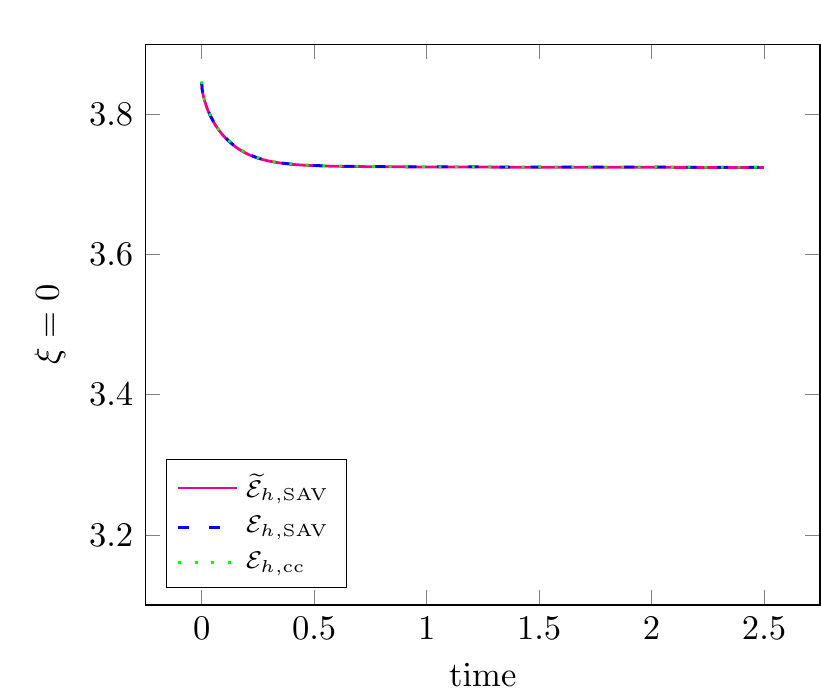}\hfill
\includegraphics[width=0.3\textwidth]{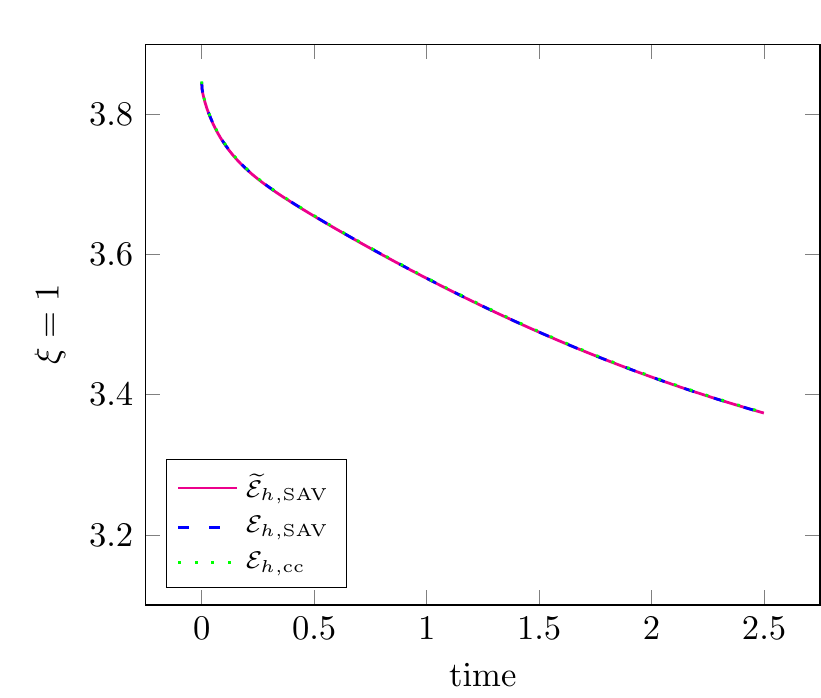}\hfill
\includegraphics[width=0.3\textwidth]{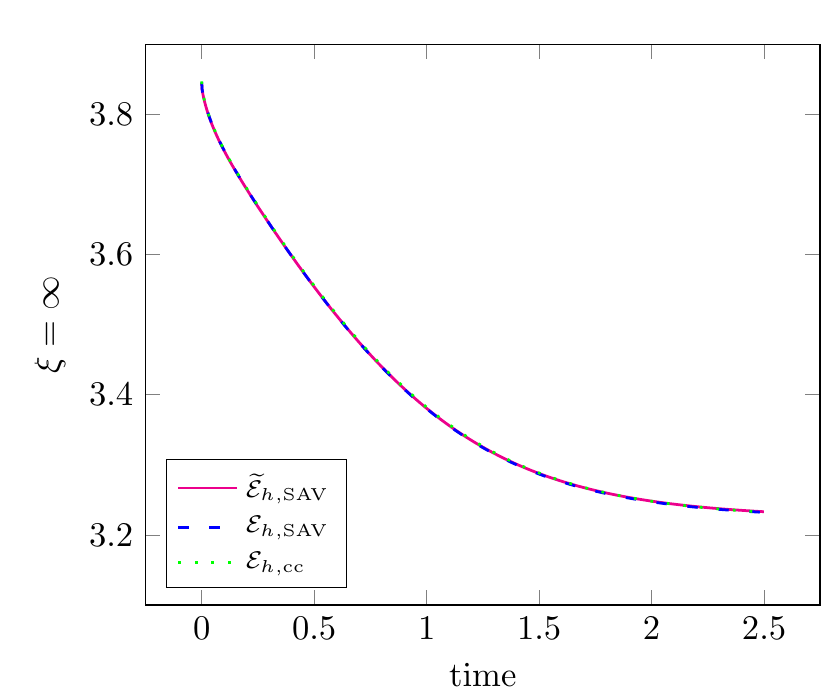}
\caption{Decrease of energy for different values of $\xi$.}
\label{fig:energy_xi}
\end{center}
\end{figure}

\ifnum\value{IMASTYLE}>1 {
\begin{table}[]
\tblcaption{Experimental order of convergence of $\phi$ and $\phi\big\vert_\Gamma$ w.r.t.~$\xi\searrow0$.}
{%
\begin{tabular}{@{}ccccccc@{}}
\tblhead{$\xi$&$\phi$&$\norm{\cdot}_{L^2\trkla{0,T;L^2\trkla{\Omega}}}$ & $\text{EOC}_\xi^\Omega$ &$\phi\big\vert_\Gamma$& $\norm{\cdot}_{L^2\trkla{0,T;L^2\trkla{\Gamma}}}$ & $\text{EOC}_\xi^\Gamma$}
$0.0001$\phz& & 3.11E-04 & - && 3.54E-04 & -\\ 
$0.0002$\phz& & 6.20E-04 & 1.00 & &7.07E-04 & 1.00\\
$0.0003$\phz& & 9.29E-04 & 1.00 && 1.06E-03 & 1.00 \\
$0.0004$\phz& & 1.24E-03 & 1.00 && 1.41E-03 & 1.00\\
$0.0005$\phz& & 1.54E-03 & 0.99 && 1.76E-03 & 1.00\\
$0.00075$& & 2.31E-03 & 0.99 && 2.64E-03 & 0.99\\
$0.001$\phzz& & 3.07E-03 & 0.99 && 3.50E-03 & 0.99\\
$0.01$\phzzz& & 2.72E-02 & 0.95 && 3.21E-02 & 0.96\\
$0.1$\phzzz\phz& & 1.40E-01 & 0.71 && 1.96E-01 & 0.77\\
$1\phantom{.}$\phzzz\phzz& & 3.73E-01 & 0.42&& 6.09E-01 & 0.49
\lastline
\end{tabular}
}
\label{tab:EOC:xi0}
\end{table}
} \else {
\begin{table}
\begin{center}
\subfloat[][EOC for $\phi$.]{
\begin{tabular}{l|cc}
 \diagbox{$\xi^{\phantom{-1}}$}{$\phi\phantom{\big\vert_\Gamma}$}& $\norm{\cdot}_{L^2\trkla{0,T;L^2\trkla{\Omega}}}$ & $\text{EOC}_\xi^\Omega$ \\
 \hline
$0.0001$ & 3.11E-04 & - \\ 
$0.0002$ & 6.20E-04 & 1.00 \\
$0.0003$ & 9.29E-04 & 1.00 \\
$0.0004$ & 1.24E-03 & 1.00 \\
$0.0005$ & 1.54E-03 & 0.99 \\
$0.00075$ & 2.31E-03 & 0.99 \\
$0.001$ & 3.07E-03 & 0.99 \\
$0.01$ & 2.72E-02 & 0.95\\
$0.1$ & 1.40E-01 & 0.71\\
$1$ & 3.73E-01 & 0.42 \\
\end{tabular}
}\hfill
\subfloat[][EOC for $\phi\big\vert_{\Gamma}$.]{
\begin{tabular}{l|cc}
 \diagbox{$\xi^{~}$}{$\phi\big\vert_{\Gamma}$}& $\norm{\cdot}_{L^2\trkla{0,T;L^2\trkla{\Gamma}}}$ & $\text{EOC}_\xi^\Gamma$ \\
 \hline
$0.0001$ & 3.54E-04 & - \\ 
$0.0002$ & 7.07E-04 & 1.00 \\
$0.0003$ & 1.06E-03 & 1.00 \\
$0.0004$ & 1.41E-03 & 1.00 \\
$0.0005$ & 1.76E-03 & 1.00 \\
$0.00075$ & 2.64E-03 & 0.99 \\
$0.001$ & 3.50E-03 & 0.99 \\
$0.01$ & 3.21E-02 & 0.96\\
$0.1$ & 1.96E-01 & 0.77\\
$1$ & 6.09E-01 & 0.49 \\
\end{tabular}
}
\caption{Experimental order of convergence for $\xi\searrow0$.}
\label{tab:EOC:xi0}
\end{center}
\end{table}
}\fi

\ifnum\value{IMASTYLE}>1 {
\begin{table}[]
\tblcaption{Experimental order of convergence of $\phi$ and $\phi\big\vert_\Gamma$ w.r.t.~$\xi^{-1}\searrow0$.}
{%
\begin{tabular}{@{}ccccccc@{}}
\tblhead{$\xi^{-1}$&$\phi$&$\norm{\cdot}_{L^2\trkla{0,T;L^2\trkla{\Omega}}}$ & $\text{EOC}_{\xi^{-1}}^\Omega$ &$\phi\big\vert_\Gamma$& $\norm{\cdot}_{L^2\trkla{0,T;L^2\trkla{\Gamma}}}$ & $\text{EOC}_{\xi^{-1}}^\Gamma$}
$0.0001$\phz& & 2.85E-04 &  -&& 4.83E-04 & -\\ 
$0.0002$\phz && 5.69E-04 & 1.00 && 9.64E-04 & 1.00\\
$0.0003$\phz && 8.52E-04 & 1.00 && 1.44E-03 & 1.00\\
$0.0004$\phz && 1.13E-03 & 1.00 && 1.92E-03 & 1.00\\
$0.0005$\phz& & 1.42E-03 & 0.99 && 2.40E-03 & 0.99\\
$0.00075$ && 2.12E-03 & 0.99 && 3.59E-03 & 0.99\\
$0.001$\phzz && 2.81E-03 & 0.99 && 4.77E-03 & 0.99\\
$0.01$\phzzz && 2.54E-02 & 0.96&& 4.31E-02 & 0.96\\
$0.1$\phzzz\phz && 1.54E-01 & 0.78&&2.64E-01 & 0.79\\
$1\phantom{.}$\phzzz\phzz && 4.02E-01 & 0.42 &&7.10E-01 & 0.43
\lastline
\end{tabular}
}
\label{tab:EOC:xiinfty}
\end{table}
} \else {
\begin{table}
\begin{center}
\subfloat[][EOC for $\phi$.]{
\begin{tabular}{l|cc}
 \diagbox{$\xi^{-1}$}{$\phi\phantom{\big\vert_\Gamma}$}& $\norm{\cdot}_{L^2\trkla{0,T;L^2\trkla{\Omega}}}$ & $\text{EOC}_{\xi^{-1}}^\Omega$ \\
 \hline 
$0.0001$ & 2.85E-04 &  -\\ 
$0.0002$ & 5.69E-04 & 1.00 \\
$0.0003$ & 8.52E-04 & 1.00 \\
$0.0004$ & 1.13E-03 & 1.00 \\
$0.0005$ & 1.42E-03 & 0.99 \\
$0.00075$ & 2.12E-03 & 0.99 \\
$0.001$ & 2.81E-03 & 0.99 \\
$0.01$ & 2.54E-02 & 0.96\\
$0.1$ & 1.54E-01 & 0.78\\
$1$ & 4.02E-01 & 0.42 \\
\end{tabular}
}\hfill
\subfloat[][EOC for $\phi\big\vert_{\Gamma}$.]{
\begin{tabular}{l|cc}
 \diagbox{$\xi^{-1}$}{$\phi\big\vert_{\Gamma}$}& $\norm{\cdot}_{L^2\trkla{0,T;L^2\trkla{\Gamma}}}$ & $\text{EOC}_{\xi^{-1}}^\Gamma$ \\
 \hline
$0.0001$ & 4.83E-04 & - \\ 
$0.0002$ & 9.64E-04 & 1.00 \\
$0.0003$ & 1.44E-03 & 1.00 \\
$0.0004$ & 1.92E-03 & 1.00 \\
$0.0005$ & 2.40E-03 & 0.99 \\
$0.00075$ & 3.59E-03 & 0.99 \\
$0.001$ & 4.77E-03 & 0.99 \\
$0.01$ & 4.31E-02 & 0.96\\
$0.1$ & 2.64E-01 & 0.79\\
$1$ & 7.10E-01 & 0.43 \\
\end{tabular} 
}
\caption{Experimental order of convergence for $\xi^{-1}\searrow0$.}
\label{tab:EOC:xiinfty}
\end{center}
\end{table}
}\fi
\bibliographystyle{amsplain}
%\bibliography{../complete}
\providecommand{\bysame}{\leavevmode\hbox to3em{\hrulefill}\thinspace}
\providecommand{\MR}{\relax\ifhmode\unskip\space\fi MR }
\providecommand{\MRhref}[2]{%
  \href{http://www.ams.org/mathscinet-getitem?mr=#1}{#2}
}
\providecommand{\href}[2]{#2}

\end{document}